\documentclass{amsart}

\usepackage{amssymb}
\usepackage{amsmath}
\usepackage{amsthm}
\usepackage{amsbsy}
\usepackage{bm}
\usepackage{hyperref}
\usepackage{cite}
\usepackage{array}
\usepackage{graphicx}
\usepackage{float}
\usepackage{upgreek}
\usepackage{color}
\usepackage{cleveref}
\usepackage{enumerate}
\usepackage{enumitem}
\usepackage{tikz-cd}
\usepackage{pgf,tikz,pgfplots}
\pgfplotsset{compat=1.18}
\usepackage{subcaption}

\date{\today}

\makeatletter
\@namedef{subjclassname@1991}{Subject}
\@namedef{subjclassname@2000}{Subject}
\@namedef{subjclassname@2010}{Subject}
\@namedef{subjclassname@2020}{Subject}
\makeatother

\theoremstyle{plain}
\newtheorem{theorem}{Theorem}
\newtheorem{lemma}[theorem]{Lemma}
\newtheorem{proposition}[theorem]{Proposition}
\newtheorem{corollary}[theorem]{Corollary}

\theoremstyle{definition}
\newtheorem{definition}[theorem]{Definition}

\newtheorem{fact}{Fact}
\newtheorem{remark}{Remark}

\newcommand {\R} {\mathbb{R}}

\def\Gam{\Gamma}

\def\<{\langle}
\def\>{\rangle}

\tikzstyle{none}=[]
\tikzstyle{new style 0}=[draw,circle,fill=white]
\tikzstyle{new edge style 1}=[draw,dashed]
\usetikzlibrary{shapes,fit,backgrounds}
\usetikzlibrary{calc}
\usetikzlibrary{positioning}

\title{Star-collision in random hypergraphs}

\author{Kartick Adhikari}
\address{Department of Mathematics, Indian Institute of Science Education and Research, Bhopal 462066}
\email{kartick [at] iiserb.ac.in}

\author{Samiron Parui}
\address{Department of Mathematics, Indian Institute of Science Education and Research, Bhopal 462066}
\email{samironparui@gmail.com}

\thanks{The research of KA was partially supported by the Inspire Faculty Fellowship: DST/INSPIRE/04/2020/000579. Additionally, KA would like to acknowledge support from the ICTP through the Associates Programme (2026-2032). SP acknowledges IISER Bhopal for financial support under the IPDF position.}


\begin{document}
	\begin{abstract}
		We study star-based symmetries in uniform hypergraphs and their consequences for matrices whose entries depend only on vertex stars. Such matrices admit a deterministic decomposition into a global component and a local component supported on equivalence classes of vertices with identical stars, known as units. While nontrivial units may exist at finite size in hypergraphs of uniformity greater than two, their persistence in random settings has remained unclear. 
		
		We analyze star collisions in random $k$-uniform hypergraphs and show that, in some particular regimes, nontrivial units disappear with high probability as the number of vertices grows. As a consequence, star-dependent matrices exhibit asymptotically trivial local structure, and their spectral behavior, invariant subspaces, and associated linear dynamics are governed by a reduced quotient object obtained by contracting vertex stars.
		
		These results identify star collisions as a finite-size phenomenon in random hypergraphs and clarify the asymptotic irrelevance of star-based symmetries for operator behavior in large random systems in particular regimes.
		
	\end{abstract}
	\maketitle
	\section{Introduction}
	In a \emph{hypergraph}, the \emph{star of a vertex} $v$ is the collection of all the hyperedges that contain $ v$. If a pair of vertices has the same star, then the symmetry causes a redundancy in the hypergraph structure. Here we study the asymptotic behaviour of these redundancies in a random $k$-uniform hypergraph across different regimes of expected vertex degrees. We refer to the phenomenon of a pair of vertices having the same star as a \emph{star collision} between the pair of vertices. Star collisions in hypergraphs are related to spectra of matrices associated with hypergraphs and invariant subspaces of these matrices. Thus, our exploration also encompasses spectral, dynamical, and reconstructive properties of hypergraphs and their associated matrices. 
	
	For a star collision between a pair of vertices $\{u,v\}$, there exists a collection of hyperedges $E$, such that both the stars of  $u$, and $v$ are equal to $E$. We refer to this collection $E$ as the \emph{support} of the star collision.
	A star collision between a pair of vertices $\{u,v\}$ is referred to as a \emph{degenerate star collision} if its support is the empty set. That is, both the vertices are isolated vertices. Our probabilistic analysis is carried out in the framework of random $k$-uniform hypergraphs generated under the independent edge model\cite{erdos1959random,gilbert1959random,johansson2008factors,rodl2007ramsey,adhikari2025spectrum,adhikari2025diameter} which is a generalization of the Erd\H{o}s--R\'enyi model \cite{erdos1959random, gilbert1959random}. 
	It is well known that if the \emph{expected degree} in a random $k$-uniform hypergraph is on the scale $\log n+c$ for a $c\le\infty$, then the number of isolated vertices follows a Poisson($e^{-c}$) distribution (\cite{burghart2024hitting,frieze2025threshold}). Thus, with that scaling of expected vertex degree, the count of degenerate star collisions follows a distribution $\binom{Z}{2}$, where $Z$ follows Poisson($e^{-c}$). See Theorem~\ref{thm:degenerate-threshold}. Moreover, we show that all the non-degenerate star collisions disappear with high probability in this regime. See Theorem~\ref{thm:star-collision-critical}. Therefore, we explore a sparser regime to study non-degenerate star collisions. We observed that when the expected vertex degree is $\frac{1}{2}(\log n+\log\log n)$, star collisions with minimal non-empty support appear with positive probability. See Theorem~\ref{thm:r1-poisson}. To obtain star-collision with larger support, one should move to a sparser regime when the expected vertex degrees are finite. See Proposition~\ref{prop:high-order-collisions}.
	
This phenomenon of star collision is intrinsically hypergraphic in nature. In a simple graph, equality of stars is highly rigid. If two distinct vertices have identical stars, then either both vertices are isolated, or they form an isolated edge component. Thus star collisions in graphs arise only from such degenerate local configurations. In contrast, for $k$-uniform hypergraphs with $k>2$, distinct vertices may exhibit non-degenerate star collision within genuinely higher-order interaction patterns, giving rise to nontrivial units at finite size. Beyond their intrinsic combinatorial significance, star-collisions give rise to a fundamental hypergraph structure: the notion of a \emph{unit}, a maximal cluster of vertices that exhibits pairwise non-degenerate star collision \cite{unit-1,unit-2}. We refer to a unit as \emph{non-trivial} if it contains more than $1$ vertices. In this work, we show that in the $\log n+c$ scale with high probability no nontrivial unit survives. See Corollary~\ref{cor:units-trivial}. Whereas in the scale $\frac{1}{2}(\log n+\log\log n)$, only finite numbers of non-trivial units of size 2 survive, Theorem~\ref{thm:units-structure}, and for larger units one must look into sparser regimes.
	
	Units are redundant blocks of hypergraphs due to the similarities of the incident structure of vertices in the same unit. Problems of vertex distinguishability and symmetry breaking in random discrete structures have a long history, particularly in random graphs and hypergraphs. Classical results show that random graphs are almost surely asymmetric, in the sense that their automorphism groups are trivial with high probability, and that local vertex statistics tend to become distinguishing with increasing density \cite{erdos1963asymmetric,bollobas1982asymptotic,kim2002asymmetry}. Related questions appear in studies of neighborhood reconstruction, local weak limits, and spectral limits of random matrices associated with graphs and hypergraphs \cite{aldous2004objective,bordenave2010resolvent,bordenave2013localization}. However, this line of work primarily focuses on graph automorphisms, degree-based separation, or global spectral statistics, and does not explicitly isolate the finer algebraic structure induced by star-based equivalence.
In the present work, non-degenerate star collisions give rise to vertex equivalence classes whose presence forces nontrivial automorphisms and, simultaneously, induces invariant subspaces for star-dependent operators. Our work isolates this star-based notion of symmetry, showing that it captures precisely the operator-relevant degeneracies.
Whereas in this work, non-degenerate star collisions give rise to vertex equivalence classes whose presence forces nontrivial automorphisms and, simultaneously, induces invariant subspaces for star-dependent operators. Our work isolates this star-based notion of symmetry, showing that it captures precisely the operator-relevant degeneracies.

From this perspective, the unit partition may be viewed as a rigid form of coarse graining induced by equality of local incidence profiles. The associated quotient hypergraph provides a canonical reduced representation for star-dependent operators and dynamics. The present work focuses on the emergence and asymptotic structure of these exact quotient classes in random hypergraphs. Extensions to approximate or dynamical notions of random equitable partitions remain an interesting direction for future investigation.

Here, we investigate how these star-based symmetries govern operator decompositions, dynamics, and reconstruction. See Theorem~\ref{thm:spectral-poisson}, \ref{thm:asymptotic_capture_dynamics}, and Proposition~\ref{thm:asymptotic_fingerprinting}. 
	
	The notion of a unit, central to the present work, has appeared previously in deterministic settings as a structural device for organizing vertex symmetries induced by identical stars, and related star-dependent operators have been studied in that context \cite{unit-1,unit-2}. More broadly, hypergraphs have traditionally been analyzed through tensors or hypermatrices, whose actions are inherently nonlinear and do not admit a direct spectral theory analogous to that of graphs \cite{tensor-book,tensor-uniform,ban-tensor,wang2024t,zucal2025actionconvergencegeneralhypergraphs}. A common approach, in both deterministic and random settings, is therefore to associate linear operators to a hypergraph via suitable contractions or linearizations of these tensors, leading to matrix representations such as adjacency, Laplacian, and random walk operators,  \cite{BANERJEE202182,swarup-LAA,unit-1,unit-2,gen-operator,incidence,ssm-er-hyp}. Star-dependent matrices form a natural and unifying class encompassing many of these constructions; their action depends only on vertex stars, and hence reflects the underlying hypergraph structure while remaining amenable to linear spectral analysis. The present work builds on this perspective by showing that units are precisely the structural features that control invariant subspaces and operator degeneracies for these star-dependent matrices associated with hypergraphs. See Theorem~\ref{thm:spectral-poisson} and \ref{thm:asymptotic_capture_dynamics}.
	
	Figure~\ref{fig:phase-diagram} illustrates a bird's-eye view of our results in this paper. The regimes are determined by the expected degree $\lambda_n$ of the random $k$-uniform hypergraph.
	
	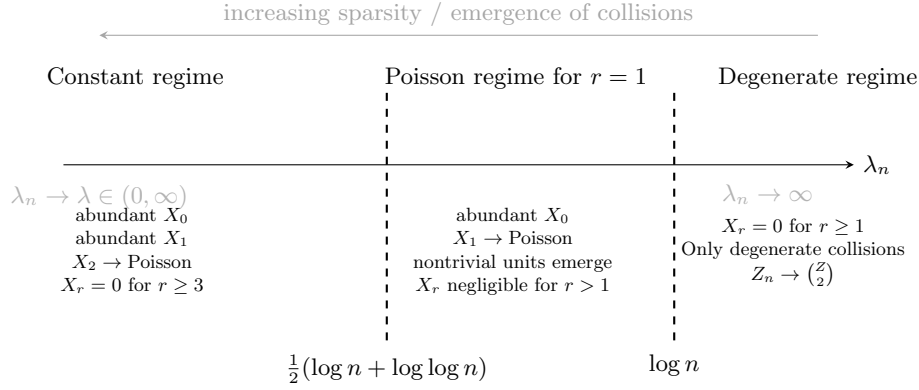
\begin{figure}[ht]
		\centering
		\begin{tikzpicture}[>=stealth, font=\small, scale=0.95]
			
			\draw[->] (2,0) -- (13,0) node[right] {$\lambda_n$};
			\draw[<- , gray!70, thin] (2.5,1.8) -- (12.5,1.8)
			node[midway, above, align=center]
			{increasing sparsity / emergence of collisions};
			\draw[dashed, thick] (6.5,1.0) -- (6.5,-2.5)
			node[below, align=center] 
			{$\tfrac{1}{2}(\log n + \log\log n)$};
			
			\draw[dashed, thick] (10.5,1.0) -- (10.5,-2.5)
			node[below, align=center] 
			{$\log n$};
			
			\node[below=4pt, gray!60] at (2.5,0) {$\lambda_n \to \lambda \in (0,\infty)$};
			\node[below=4pt, gray!60] at (11.8,0) {$\lambda_n \to \infty$};
			
			
			\node at (3.0,1.2) {Constant regime};
			
			\node at (8.3,1.2) {Poisson regime for $r=1$};
			
			\node at (12.5,1.2) {Degenerate regime};
			
			
			\node[align=center,scale=0.8] at (3.0,-1.2) {
				\begin{tabular}{c}
				abundant $X_0$\\
				abundant $X_1$\\
					$X_2 \rightarrow \mathrm{Poisson}$\\
					$X_r = 0$ for $r \ge 3$
				\end{tabular}
			};
			
			\node[align=center,scale=0.8] at (8.3,-1.2) {
				\begin{tabular}{c}
					abundant $X_0$\\
					$X_1 \rightarrow \mathrm{Poisson}$\\
					nontrivial units emerge\\
				$X_r$ negligible for $r>1$
				\end{tabular}
			};
			
			\node[align=center,scale=0.8] at (12.2,-1.2) {
				\begin{tabular}{c}
					$X_r = 0$ for $r \ge 1$\\
					Only degenerate collisions\\
					$Z_n \rightarrow \binom{Z}{2}$
				\end{tabular}
			};
			
		\end{tikzpicture}
		\caption{
			Phase diagram for star-collisions in $\mathcal{H}(n,k,p_n)$ as a function of 
			$\lambda_n=\binom{n-1}{k-1}p_n$. Moving from right to left (decreasing $\lambda_n$), progressively richer collision structures emerge. Here $X_r$ denotes the number of star-collisions with support size $r$, while $X_0$ corresponds to collisions with empty support. In the dense regime $\lambda_n=\log n+c$, all positive-support collisions vanish with high probability. At the scale 
			$\lambda_n=\frac12(\log n+\log\log n)+O(1)$,
			support-$1$ collisions appear with positive probability. In bounded expected-degree regimes, collisions with larger support sizes become observable.
		}
		\label{fig:phase-diagram}
	\end{figure}
	
	The remainder of the paper is organized as follows. In Section~\ref{sec:prelim}, we mainly recall some deterministic and probabilistic facts that we are going to use throughout our work.
	The deterministic facts include basic definitions and notation for uniform hypergraphs, vertex stars, star collisions, units, and star-dependent matrices, and some decomposition results that are relevant to our study. The probabilistic facts include the notion of random hypergraphs, related random structures, and empirical spectral distribution.  Section~\ref{sec:probabilistic} contains the main results of our work. This section contains two subsections. The first of the two is on the probabilistic analysis of star collisions in random $k$-uniform hypergraphs, identifying the regimes and mechanisms under which nontrivial units disappear with high probability. In the later Subsection, we investigate the asymptotic consequences of this collapse for star-dependent operators, including spectral behavior and invariant subspaces. Section~\ref{sec:dynamics} discusses implications for linear dynamics and reconstruction problems driven by such operators. We explored the star-based reconstruction problem in hypergraphs in Section~\ref{sec:reconstruction}. The proof of all the results in this paper is provided in Section~\ref{sec:proof}. 
	
	\section{Preliminaries}\label{sec:prelim}
	This section introduces the notation and basic objects used throughout the paper. We define vertex stars and star-collisions, describe the class of star-dependent matrices, and recall the associated decomposition induced by units. These deterministic constructions will be used later in the random setting. 
	\subsection{Hypergraphs and vertex stars}
	
	A hypergraph is a pair $H=(V(H),E(H))$, where $V(H)$ is a nonempty set of vertices and $E(H)$ is a family of nonempty subsets of $V(H)$, called hyperedges. If the cardinality $|e|=k$ for some $k\ge 2$ for each hyperedge $e\in E(H)$, then the hypergraph $H$ is referred to as $k$-uniform hypergraph. Evidently, if $H$ is a $k$-uniform hypergraph, then $E(H)\subseteq \left[\tbinom{V(H)}{k}\right]$.
	For a vertex $v\in V$, the \emph{star} of $v$, denoted by $\Gam(v)$, is the collection of hyperedges containing $v$. That is, $\Gam(v)=\{e\in E(H):v\in e\}\;\mbox{for any vertex}\;v$.
	
	Two vertices $u,v\in V$ are said to exhibit \emph{star collision} if $\Gam(u)=\Gam(v)$. There are two types of star collision- (i)\emph{degenerate star collision} between $u,v$, in which we have  $\Gam(u)=\Gam(v)=\emptyset$, and (ii)\emph{non-degenerate star collision} between $u,v$ which is characterized by  $\Gam(u)=\Gam(v)\ne \emptyset$. We refer to a star collision between a pair of vertices $\{u,v\}$ as \emph{star collision of support size $r$} if $\Gam(u)=\Gam(v)=E$ with $|E|=r$. We denote the total number of star collisions of support size $r$ in a hypergraph $H$ as $X_r(H)$. If there is no scope of confusion regarding the hypergraph $H$, we will just write $X_r$ instead of $X_r(H)$. That means total non-degenerate star collision $Y=\sum_{r\ge 1}X_r$.
	
	Non-degenerate star-collision induces an equivalence relation on $V(H)$, where $u \sim v$ if $\Gamma(u)=\Gamma(v)\neq \emptyset$. The equivalence classes under this relation are called units. 
	For any unit $W$, by definition if $|W|>1$, then there exists a sub-collection of hyperedges $E_W(\ne \emptyset)$ such that $\Gamma(v)=E_W$ for all $v\in W$. Therefore, $W\subseteq e$ for all $e\in E_W$. Thus, in a $k$-uniform hypergraph, each unit contains at most $k$ vertices. For  hypergraph $H$, we denote the collection of all the units in $H$ as $\mathfrak{U}(H)$. 
	\subsection{Linear operators associated with hypergraphs}
	Hypergraphs admit natural representations in terms of higher-order tensors, such as adjacency or Laplacian tensors in the uniform case. A substantial literature is devoted to tensor-based formulations of hypergraph dynamics and spectral theory \cite{ban-tensor,tensor-book}. While these representations are conceptually natural, the resulting spectral problems are inherently nonlinear and often difficult to analyze.
	A widely used alternative is to contract such tensors along suitable modes, thereby producing matrices indexed by the vertex set. These matrices act linearly on $\R^{V(H)}$ and are directly tied to the underlying hypergraph structure. Unlike the bijective correspondence between graphs and their adjacency matrices, hypergraph can not be reconstructed from any matrices associated with it. Though recent work shows that these contracted matrices retain rich structural and dynamical information about the hypergraph, including mixing behavior, localization phenomena, and coarse-grained dynamics \cite{ssm-er-hyp,LuPeng2011,Traversaetal2023,Pavlopoulosetal2018}. 
	In this paper, our study encompasses a family of matrices obtained via such contractions, their spectral and structural properties. Now we are going to introduce this family of matrices.
	\subsection{Star-dependent matrices}
	Star-dependent matrices are specific contracted matrices associated with hypergraphs which are controlled by the stars of the hypergraphs.
	\begin{definition}\label{def:star-dep}
		A square matrix $M=(M_{uv})_{u,v\in V(H)}$ indexed by the vertex set of a hypergraph $H$ is called \emph{star-dependent} if there exist functions $F$ and $G$ such that
		$
		M_{uv}=F\bigl(\Gam(u),\Gam(v)\bigr) \quad \text{for all } u\neq v,
		$
		while the diagonal entry,
		$
		M_{uu}=G\bigl(\Gam(u)\bigr)$ $ \text{for all } u\in V(H).
		$
	\end{definition}
	Thus, 
	star-dependent matrices are not arbitrary vertex-indexed matrices but are rigidly constrained by the hypergraph star incidence pattern.
	Many standard matrices associated with hypergraphs fall into this class, including various adjacency matrices, Laplacian-type operators, and transition matrices governing hypergraph random walks (see for instance \cite{unit-1,unit-2}). For example, in \cite{bretto2013hypergraph}, one adjacency matrix associated with a hypergraph is considered such that for a pair of distinct vertices $u,v$, the $(u,v)$-th entry is the co-degree of $u,v$, that is, $|\Gamma(u)\cap \Gamma(v)|$. In \cite{BANERJEE202182}, the $(u,v)$-th entry of the adjacency matrix is $\sum_{e\in\Gamma(u)\cap\Gamma(v)}\frac{1}{|e|-1}$. Many known Laplacian matrices, signless Laplacian and probability transition matrices are such that their $(u,v)$-th entry can be expressed as a function of stars of these vertices \cite{gen-operator,rodriguez2003laplacian,chierichetti2015hypergraph}. The results of this paper apply uniformly to all such operators.
	
	\subsection{Units, equitable partitions, and unit contraction}
	\label{section-prili-unit}
	Let $M=(M_{uv})_{u,v\in V}$ be a real square matrix indexed by a finite set $V$. A partition
	$
	\mathfrak P=\{P_1,\dots,P_k\}
	$
	of $V$ is called an \emph{equitable partition} for $M$ if, for every $i,j$ and every pair $u,v\in P_i$, one has
	$\sum_{w\in P_j} M_{uw} = \sum_{w\in P_j} M_{vw}=\beta_{ij}$.
	In this case, the common value $\beta_{ij} $ defines the $(i,j)$ entry of a $k\times k$ matrix, denoted by $M/\mathfrak P$ and called the \emph{quotient matrix} associated with $\mathfrak P$.
	
	Equivalently, let $\mathcal C(\mathfrak P)\subset \R^V$ denote the subspace of vectors that are constant on each part of the partition $\mathfrak P$. The partition $\mathfrak P$ is equitable for $M$ if and only if $\mathcal C(\mathfrak P)$ is invariant under the action of $M$. Under the natural identification $\mathcal C(\mathfrak P)\cong \R^k$, the restriction of $M$ to this invariant subspace is represented by the quotient matrix $M/\mathfrak P$. See \cite[Section-9.3]{Godsil-AGT}.
	
	Now let $H(V(H),E(H))$ be a hypergraph and let $M$ be a star-dependent matrix on $H$. Recall that vertices are grouped into \emph{units} according to star-equivalence. Because the entries of $M$ depend only on vertex stars, the partition of $V(H)$ into units is always an equitable partition for $M$. We refer to the associated quotient matrix as the \emph{unit contraction} of $M$ and denote it by $\widehat M$.
	
	If $\alpha$ is an eigenvalue of the unit contraction $\widehat M$ with eigenvector
	$\hat f\in \R^{|\mathfrak U(H)|}$, we define its \emph{unit-lift} to $V(H)$ as the vector
	$f\in \R^{V(H)}$ obtained by assigning to each vertex the value of $\hat f$ on its unit.
	By construction, $f$ is constant on units and satisfies
	$
	Mf=\alpha f.
	$
	We therefore refer to $f$ as a \emph{lifted eigenvector}, and to $\alpha$ as an eigenvalue of $M$ lifted from the unit contraction. Suppose that $\mathcal H_{\mathrm{glob}} $ is the vector space generated by all the lifted eigenvectors of $M$.
	We summarize the consequences of these constructions for star-dependent matrices as facts, which we state without proof. For detail proof we refer the reader to \cite{unit-1,unit-2}. 
	
	This equitable structure allows one to compress the action of $M$ to the quotient space of units while preserving part of the spectrum through lifted eigenvectors.
	\begin{fact}[Units induce equitable partitions and unit-lifts]
		Let $M$ be a star-dependent matrix on a hypergraph $H$. The partition of $V(H)$ into units is an equitable partition for $M$. Every eigenvalue of the unit contraction $\widehat M$ is also an eigenvalue of $M$, with eigenvectors given by unit-lifts of the eigenvectors of $\widehat M$.
	\end{fact}
	Suppose that $u\in V(H)$ and $\mathbf{1}_u:V(H)\to\{1,0\}$ is a vector with its $u$-th entry is $1$ and all other entries are $0$. For vertices belonging to the same non-trivial unit, the corresponding rows of a star-dependent matrix are identical up to the interchange of the entries indexed by those vertices. Consequently, taking the difference of the corresponding coordinate vectors produces a localized eigenvector.
	\begin{fact}\label{fact-eigen-unit}
		Let $M$ be a star-dependent matrix on a hypergraph $H$ and $W$ is a non-trivial unit in $H$ then, for distinct $u,v\in W$, $(m_{uu}-m_{uv})$ is an eigenvalue with eigenvector $\mathbf{1}_u-\mathbf{1}_v$.
	\end{fact}
	Corresponding to a non-trivial unit, the eigenvalue $(m_{uu}-m_{uv})$ is referred to as \emph{unit-eigenvalue}. Any non-trivial unit $W$ corresponds to a collection of eigenvectors $Y_W=\{\mathbf{1}_u-\mathbf{1}_v:\{u,v\}\subseteq W\}$ of the eigenvalue $(m_{uu}-m_{uv})$. Consequently, the subspace $S_W=\{x:V(H)\to\mathbb{R}:\sum_{v\in W}x(v)=0, x(w)=0~\forall~ w\notin W\}$ generated by the collection $Y_W$ is an invariant subspace of the star-dependent matrix $M$. The dimension of $S_W$ is $|W|-1$. Therefore, the direct sum $\bigoplus_{W\in\mathcal{W}}S_W$, where $\mathcal{W}$ is a the collection of non-trivial unit in $H$, is also an invariant subspace of the star-dependent matrix $M$. We denote the invariant subspace as $\mathcal H_{\mathrm{loc}}=\bigoplus_{W\in\mathcal{W}}S_W$.

	The above two facts imply a natural orthogonal decomposition of $\R^{V(H)}$ that would be discussed in the following fact. For details see
	 \cite{unit-1,unit-2}.
	\begin{fact}[Finite-size spectral decomposition]
		\label{thm:finite_decomposition}
		The space $\R^{V(H)}$ admits an orthogonal decomposition
		$
		\R^{V(H)} = \mathcal H_{\mathrm{glob}} \oplus \mathcal H_{\mathrm{loc}},
		$
		where $\mathcal H_{\mathrm{glob}}$ consists of functions that are constant on
		each unit, and $\mathcal H_{\mathrm{loc}}$ is its orthogonal complement.
		Both subspaces are invariant under a star-dependent matrix $M$. The restriction of $M$ to $\mathcal H_{\mathrm{glob}}$ is unitarily equivalent to the unit contraction $\widehat M$. Every eigenvalue of $\widehat M$ is an eigenvalue of $M$, with eigenvectors obtained by unit-lifting eigenvectors of $\widehat M$. The remaining eigenvalues of $M$ are precisely the eigenvalues of $M$ restricted to $\mathcal H_{\mathrm{loc}}$, which consist entirely of unit-eigenvalues.
	\end{fact}
	\begin{remark}
		Throughout this work, units are defined using \emph{non-degenerate} star-collisions. 
		If degenerate collisions were also included, all isolated vertices—having identical empty stars—would collapse into a single equivalence class.
		
		Under our definition, each isolated vertex forms its own singleton unit. 
		This distinction is essential for dynamical considerations: isolated vertices, lacking shared incidence structure, do not exhibit synchronization and should not be grouped together. If the hypergraph is connected, isolated vertices do not occur, and the two definitions coincide.
	\end{remark}
	\subsection{Random hypergraphs and random star-dependent matrices}
	
	We now make the random setting precise. Throughout, let
	$H_n$ denote a hypergraph with vertex set $V(H)$,
	$|V(H)| = n$. In that case, we will assume $V(H)=[n]$.
	
	A $k$-uniform \emph{random hypergraph} is defined as follows. For each subset
	$e \in \mathcal{V}_k$, let $\xi_e$ be an independent
	Bernoulli random variable with parameter $p = p_n \in (0,1)$.
	We declare $e$ to be a hyperedge if and only if $\xi_e = 1$.
	
	The family
	$
	\{\xi_e : e \in \mathcal{V}_k\}
	$
	generates a product (cylindrical) probability space
	$\mathcal{H}(n,k,p_n)=(\Omega_n, \mathcal{F}_n, \mathbb{P})$.
	A realization $\omega \in \Omega_n$ determines a hyperedge set
	$E(H)(\omega)
	=
	\{ e \subseteq V(H) : |e| =k,\ \xi_e(\omega) = 1 \},$
	and hence a hypergraph-valued random variable
	$
	H_n(\omega) = (V(H), E(H)(\omega)).
	$ Henceforth, a realization $H_n(\omega)$ will simply be denoted by
	$
	H_n\in\mathcal H(n,k,p_n).
	$ Throughout the paper, all random variables and events are defined with respect to the probability space
	$\mathcal{H}(n,k,p_n)$.
	Consequently, their dependence on $n$ will be suppressed from the notation whenever no ambiguity arises. All asymptotic statements, limits, and probabilistic convergences are understood with respect to the limit $n\to\infty$, unless explicitly stated otherwise.
	This construction is a natural generalization of the Erd\H{o}s--R\'enyi
	random graph (\cite{erdos1959random,gilbert1959random}) to the hypergraph setting \cite{rodl2007ramsey}.
	For a fixed vertex $v \in V(H)$, the star
	$
	\Gamma(v) = \{ e \in E(H) : v \in e \}
	$
	is a random object induced by the randomness of $H_n$.
	More precisely, $\Gamma(v)$ is a measurable mapping from $\Omega_n$
	into the space of all possible stars rooted at $v$.
	Thus, vertex stars are star-valued random variables.
	
	As a consequence, star-equivalence defines a random equivalence relation
	on $V(H)$, and the induced partition of $V(H)$ into units is itself a
	random partition. Both the number of units and their sizes depend on the
	realization of the random hypergraph.
	
	Note that, for $H_n\in \mathcal{H}(n,k,p_n)$ and $v\in V(H_n)$, the degree of $ v$ can be represented as the sum of all $\xi_e$ such that $v\in e$. Since there are total $\tbinom{n-1}{k-1}$ such $k$-sets that contains $v$, the expected degree is $p_n \tbinom{n-1}{k-1}$. Thus, a natural parameter associated with this random hypergraph model is
	the expected vertex degree,
	$
	\lambda_n := p_n \tbinom{n-1}{k-1},
	$
	which measures the expected size of the star of a single vertex and provides the
	standard notion of sparsity in random $k$-uniform hypergraphs. This construction is a widely adopted generalization of the Erd\H{o}s--R\'enyi model \cite{erdos1959random, gilbert1959random} to the hypergraph setting \cite{johansson2008factors, rodl2007ramsey}, where the parameter $\lambda_n$ serves as the standard measure of sparsity in both spectral and probabilistic analyses \cite{stephan2024sparse, LuPeng2011, zhou2021sparse,adhikari2025spectrum}.
\subsection{Empirical spectral distributions}

Let $M$ be an $N\times N$ self-adjoint matrix with eigenvalues
$\alpha_1,\dots,\alpha_N$ counted with multiplicity. Its empirical spectral
distribution is
$\mu_M=\frac1N\sum_{i=1}^N \delta_{\alpha_i}$.
This is a probability measure on $\mathbb R$. For a sequence $(M_n)$ of
self-adjoint matrices, we say that $\mu_{M_n}$ converges weakly to a probability
measure $\nu$ on $\mathbb R$, denoted by
$\mu_{M_n}\rightsquigarrow \nu$, if for every bounded continuous function $f$,
$$
\int_{\mathbb R} f(x)\,d\mu_{M_n}(x)\to \int_{\mathbb R} f(x)\,d\nu(x).
$$
When this occurs, $\nu$ is called the limiting spectral distribution.

	\section{Main Result}\label{sec:probabilistic}
	This section develops the probabilistic theory of star-dependent structures in random hypergraphs and shows how it governs the behavior of associated linear operators. 
	Our results reveal a hierarchy of probabilistic regimes in the geometry of vertex stars, inducing corresponding changes in the combinatorial, spectral, and dynamical properties of star-dependent constructions.
	The transition is controlled by the scaling of
	$
	\lambda_n = \binom{n-1}{k-1} p_n,
	$
	the expected degree of a vertex. At the scale $\lambda_n = \log n + c$, all non-trivial star-collisions disappear with high probability, and the induced unit structure becomes trivial.
	
	In contrast, at the finer scale
	$
	\lambda_n = \tfrac{1}{2}(\log n + \log\log n) + w,
	$
	 star-collisions of positive sized support persist with positive probability. This leads to the appearance of a finite random collection of non-trivial units, whose statistics are governed by a Poisson law, and which give rise to a finite-dimensional local correction to the global structure.
	
	We first describe this transition at the level of star-collisions and unit structure, and then show how it propagates to spectral properties, invariant subspaces, and linear dynamical systems associated with star-dependent matrices.
	\subsection{Star-collision structure across regimes}\label{sec:main result-1}
	We now analyze the structure of star-collisions in the random $k$-uniform hypergraph 
	$H_n \in \mathcal{H}(n,k,p_n)$ across the regimes introduced above, distinguishing between degenerate and non-degenerate configurations and identifying the scales at which non-trivial collisions first appear.
	Recall that $X_r$ denotes the number of star-collisions of support size $r$, 
	that is, unordered pairs $\{u,v\}$ such that $\Gamma(u)=\Gamma(v)$ 
	and $|\Gamma(u)|=r$.
	\begin{theorem}\label{thm:star-collision-critical}
		Fix $k \ge 2$ and $c \in \mathbb{R}$. Let $X_r$ denote the number of 
		star-collisions of support size $r$ in $H_n \in \mathcal{H}(n,k,p_n)$. If
		$\lambda_n = \log n + c,$
		then the following hold:
		\begin{enumerate}[leftmargin=2em]
			\item[(i)] For every fixed $r \ge 1$,
			$\mathbb{P}(X_r \ge 1) \longrightarrow 0,
			\;
			\text{and in particular }
			\mathbb{E}[X_r] \longrightarrow 0.$
			\item[(ii)] The number of degenerate star-collisions satisfies
			$\mathbb{E}[X_0] \longrightarrow \frac{1}{2} e^{-2c}.$
		\end{enumerate}
	\end{theorem}
	Theorem~\ref{thm:star-collision-critical} shows that in the regime
	$\lambda_n = \log n + c$, all non-degenerate star-collisions vanish
	with high probability.
	
	\begin{corollary}[Absence of non-trivial units in the critical window]
		\label{cor:units-trivial}	Suppose that $k \ge 2$ and $\mathfrak U(H_n)$ is the collection of all the unit in $H_n \in \mathcal{H}(n,k,p_n)$.
		In the regime
		$\lambda_n = \log n + c,$
		with $c \in \mathbb{R}$ fixed, we have
		$\mathbb{P}(\exists\, W \in \mathfrak U(H_n) \text{ with } |W| \ge 2) \longrightarrow 0.$
	\end{corollary}
In particular Corollary~\ref{cor:units-trivial} shows that in the regime
$\lambda_n = \log n + c$, with high probability, every unit is a singleton.
We now describe the limiting behavior of $X_0$ in this regime.
	e now describe the limiting behavior of $X_0$ in this regime.
	\begin{theorem}[Sharp threshold for degenerate star-collisions]
		\label{thm:degenerate-threshold}
		Let $ X_0$ denote the number of degenerate star-collisions in 
		$H_n \in \mathcal{H}(n,k,p_n)$, and let
		$\lambda_n = \tbinom{n-1}{k-1} p_n.$
		Then $\lambda_n = \log n$ is the sharp threshold for the existence 
		of degenerate star-collisions, in the sense that:
		
		\begin{enumerate}
			\item[(i)] If $\lambda_n = \log n + w(n)$ with $w(n) \to +\infty$, then
			$\mathbb{P}(X_0 = 0) \longrightarrow 1.$

			\item[(ii)] If $\lambda_n = \log n - w(n)$ with $w(n) \to +\infty$, then
			$\mathbb{P}(X_0 > 0) \longrightarrow 1.$
			
		\end{enumerate}
	\end{theorem}
	Theorem~\ref{thm:star-collision-critical} and 
	Theorem~\ref{thm:degenerate-threshold} together identify $\lambda_n = \log n$
	as the threshold at which degenerate star-collisions emerge and dominate the collision structure.
	\begin{theorem}[Degenerate star collisions in the critical window]
		\label{thm:degenerate-count}
		Let $k \ge 2$ be fixed and let $ X_0$ denote the number of 
		degenerate star-collisions in $H_n \in \mathcal{H}(n,k,p_n)$.
		Assume that
		$\lambda_n = \log n + c,
		\; c \in \mathbb{R}.$
		Then
		$X_0 \xrightarrow{d} \tbinom{Z}{2},
		\; \text{where } Z \sim \mathrm{Poisson}(e^{-c})$.
	\end{theorem}
	This limit expresses the degenerate collision structure as a quadratic functional of a Poisson random variable, and is therefore non-Poisson.
	To capture non-degenerate star-collisions, we consider the finer scale
	$\lambda_n = \tfrac{1}{2}(\log n + \log\log n) + w$, $w\in\mathbb{R}$.
	At this scale, star-collisions of minimal positive support occur with non-vanishing probability, while higher-support configurations remain negligible.
	\begin{theorem}[Poisson law for $1$–support star collisions]
		\label{thm:r1-poisson}
		Let $k \ge 2$ be fixed and let $H_n \in \mathcal{H}(n,k,p_n)$. 
		Let $X_1$ denote the number of star-collisions of support size $1$.
		Assume that
		$\lambda_n 
		= 
		\tfrac{1}{2}(\log n + \log\log n) + w,
		\; w \in \mathbb{R} \text{ fixed}.$
		Then
		$X_1 \xrightarrow{d}
		\mathrm{Poisson}\!\left(
		\frac{k-1}{4}\, e^{-2w}
		\right)$.
	\end{theorem}
	This shows that star-collisions of support size $1$ constitute the dominant non-degenerate contribution in this regime.
	\begin{proposition}[Higher-order star collisions]
		\label{prop:high-order-collisions}
		Let $k \ge 2$ be fixed and let $H_n \in \mathcal{H}(n,k,p_n)$. 
		For each integer $r \ge 2$, let $X_r$ denote the number of 
		star-collisions of support size $r$.
		
		\begin{enumerate}
			\item[(i)] If $r = 2$ and $\lambda_n \to \lambda \in (0,\infty)$, then
			$ X_2 \xrightarrow{d}
			\mathrm{Poisson}\!\left(
			\frac{(k-1)^2}{4}\,\lambda^2 e^{-2\lambda}
			\right).$

			\item[(ii)] For every fixed $r \ge 3$ and  $\lambda_n \to \lambda \in (0,\infty)$,
			$\mathbb{P}(X_r \ge 1) \longrightarrow 0$.
			
		\end{enumerate}
	\end{proposition}
	Thus, collisions of support size $2$ arise in this scaling regime, whereas all higher-support collisions vanish asymptotically.
	\begin{theorem}[Structure of non-trivial units]
		\label{thm:units-structure}
		Let $H_n \sim \mathcal{H}(n,k,p_n)$ and for each $m \ge 2$, let $U_m$ denote the number of units of size $m$,
		and let $Y := \sum_{m\ge2} U_m$ be the total number of non-trivial units. Assume
		$\lambda_n  = \tfrac{1}{2}(\log n + \log\log n) + w,
		\; w\in\mathbb{R}$.
		Then the following hold.
		
		\begin{enumerate}
			\item[(i)] 
			$\mathbb{P}(U_m = 0 \text{ for all } m \ge 3) \to 1.$
			\item[(ii)] 
			The total number of non-trivial units satisfies
			$Y \xrightarrow{d}
			\mathrm{Poisson}\!\left(
			\frac{k-1}{4}\, e^{-2w}
			\right).$
		\end{enumerate}
	\end{theorem}
	In particular, asymptotically all non-trivial units are of size $2$.
	\subsection{Linear-algebraic and dynamical consequences of the unit decomposition}
	\label{sec-appl}
	The results of the previous subsection show that the unit structure of $H_n$
	undergoes a sharp transition across regimes, ranging from complete triviality
	to a finite random collection of non-trivial units.
	We now examine how this structure constrains star-dependent operators
	associated with $H_n$. In particular, we show that the unit decomposition
	governs spectral properties, invariant subspaces, and dynamical behavior.
	\subsubsection{Spectral consequences}
	\label{sec:application-spec}
	We now describe the spectral implications of the star-collision structure established in Section~\ref{sec:main result-1}. 
	Let $M$ be a star-dependent matrix associated with $H_n \sim \mathcal{H}(n,k,p_n)$. 
	Recall from Fact~\ref{thm:finite_decomposition} that $\mathbb{R}^{V(H_n)}$ admits an invariant decomposition
	$\mathbb{R}^{V(H_n)} = \mathcal H_{\mathrm{glob}} \oplus \mathcal H_{\mathrm{loc}},$
	and that the spectrum of $M$ splits into contributions from these two subspaces.
	The dimension of the local subspace is determined by the unit structure via
	$\dim(\mathcal H_{\mathrm{loc}}) = \sum_{W} (|W|-1)$, where the sum runs over all non-trivial units $W$ in $H_n$. 
	Thus, our results on the number and size of units translate directly into spectral information.
	In the regime $\lambda_n = \log n + c$, Corollary~\ref{cor:units-trivial} implies that, with high probability, there are no non-trivial units, and hence $\mathcal H_{\mathrm{loc}} = \{0\}$. 
	In contrast, in the regime
	$\lambda_n = \tfrac{1}{2}(\log n + \log\log n) + w$,
	Theorem~\ref{thm:units-structure} shows that the number of non-trivial units converges in distribution to a Poisson random variable and that, with probability tending to one, all such units have size $2$. Consequently, $\dim(\mathcal H_{\mathrm{loc}})$ converges in distribution to a Poisson law.
	We now state the resulting spectral description.
	
	\begin{theorem}[Spectral structure across regimes]
		\label{thm:spectral-poisson}
		Let $k \ge 2$ be fixed and let $M$ be a star-dependent matrix associated with 
		$H_n \sim \mathcal{H}(n,k,p_n)$.
		
		\begin{enumerate}[label=(\roman*), leftmargin=2em]
			
			\item If $\lambda_n = \log n + c$, then
			$\mathbb P\big(\mathrm{Spec}(M) = \mathrm{Spec}(\widehat M)\big) \longrightarrow 1.$
			\item If $\lambda_n = \tfrac{1}{2}(\log n + \log\log n) + w,$
			then
			\begin{enumerate}
				\item[(a)] 
				$\dim(\mathcal H_{\mathrm{loc}})
				\xrightarrow{d}
				\mathrm{Poisson}\!\left(\frac{k-1}{4}\,e^{-2w}\right)$,
				
				\item[(b)] the spectrum of $M$ consists of the spectrum of its unit contraction $\widehat M$ together with a random finite multiset of additional eigenvalues, whose total multiplicity converges in distribution to
				$\mathrm{Poisson}\!\left(\frac{k-1}{4}\,e^{-2w}\right)$.
				
			\end{enumerate}
			
		\end{enumerate}
	\end{theorem}
	
	The additional eigenvalues described above arise from non-trivial units. 
	In the critical window, these units are, with high probability, of size $2$, so that each unit contributes exactly one eigenvalue.
	
	As an immediate consequence, the local component $M|_{\mathcal H_{\mathrm{loc}}}$ has random finite rank. Standard results on finite-rank perturbations therefore imply that the empirical spectral distribution of $M$ is asymptotically equivalent to that of its global component.
	
	\begin{corollary}[Stability of the empirical spectral distribution]
		\label{cor:esd-stability}
		If either $\lambda_n = \log n + c$ or  $\lambda_n = \tfrac{1}{2}(\log n + \log\log n) + w$, $c,w\in\mathbb{R}$,
		$\sup_n \mathbb{E}\!\left[\frac{1}{n}\operatorname{Tr}(M_n M_n^*)\right] < \infty$, and $M_n$ is a sequence of self-adjoint star-dependent matrices
		then the empirical spectral distribution of $M_n$ converges weakly if and only if the empirical spectral distribution of its restriction to $\mathcal H_{\mathrm{glob}}$ converges, and both limits coincide.
	\end{corollary}
	\subsubsection{Invariant subspaces of star-dependent matrices and dynamical implications}
	\label{sec:dynamics}
	
	We now describe how the unit structure of a hypergraph constrains the invariant subspaces of associated star-dependent matrices. 
	Let $M$ be a linear operator acting on $\mathbb R^{V(H)}$. 
	A subspace $\mathcal U \subseteq \mathbb R^{V(H)}$ is said to be \emph{$M$-invariant} if 
	$M u \in \mathcal U$ for all $u \in \mathcal U$. 
	We denote by $\mathcal I(M)$ the collection of all $M$-invariant subspaces of $\mathbb R^{V(H)}$.
	
	Recall from Section~\ref{section-prili-unit} that for any star-dependent matrix $M$, the space $\mathbb R^{V(H)}$ admits an invariant decomposition
	$\mathbb R^{V(H)} = \mathcal H_{\mathrm{glob}} \oplus \mathcal H_{\mathrm{loc}}$,
	where $\mathcal H_{\mathrm{glob}}$ consists of functions that are constant on units and $\mathcal H_{\mathrm{loc}}$ captures fluctuations within units.
	Theorem~\ref{thm:units-structure} shows that, in the random hypergraph setting, the dimension of $\mathcal H_{\mathrm{loc}}$ is either zero with probability tending to one or remains finite and random.
	We now show that this structural dichotomy governs the full lattice of invariant subspaces of $M$.
	
	\begin{theorem}[Asymptotic structure of invariant subspaces]
		\label{thm:asymptotic_capture_dynamics}
		Let $H_n\sim\mathcal H(n,k,p_n)$ and let $M_n$ be a star-dependent matrix
		associated with $H_n$, with unit contraction $\widehat M_n$.
		
		\begin{enumerate}[label=(\roman*), leftmargin=2em]
			
			\item If $\lambda_n=\log n+c$, then
			$\mathbb P\!\left(
			\mathcal I(M_n)
			=
			\mathrm{Lift}\bigl(\mathcal I(\widehat M_n)\bigr)
			\right)
			\longrightarrow 1.$
			\item If
			$\lambda_n=\tfrac{1}{2}(\log n+\log\log n)+w$,
			then
			every $M_n$-invariant subspace $\mathcal V \subseteq \mathbb R^{V(H_n)}$
			admits a decomposition
			$\mathcal V
			=
			\mathcal V_{\mathrm{glob}}
			\oplus
			\mathcal V_{\mathrm{loc}}$,
			where $\mathcal V_{\mathrm{glob}}\subseteq \mathcal H_{\mathrm{glob}}$ and
			$\mathcal V_{\mathrm{loc}} \subseteq \mathcal H_{\mathrm{loc}}$.
			Moreover, for every $\varepsilon>0$ there exists $s>0$ such that
			\[
			\limsup_{n\to\infty}
			\mathbb P\!\big(\dim(\mathcal V_{\mathrm{loc}})>s\big)<\varepsilon.
			\]
		\end{enumerate}
	\end{theorem}
	
	We now interpret these structural results in the context of linear dynamics. 
	Given a star-dependent matrix $M$, consider the discrete-time linear dynamical system
	\[
	x(t+1)=M\,x(t), \qquad x(t)\in \mathbb R^{V(H)}.
	\]
	A subset $C \subseteq V(H)$ is said to exhibit \emph{cluster synchronization} if all coordinates of $x(t)$ indexed by $C$ remain equal for all $t$, whenever they are equal at time $t=0$.
	We say that a state is \emph{unit-synchronized} if it is constant on each non-trivial unit of $H$, that is, if it belongs to $\mathcal H_{\mathrm{glob}}$. 
	Since $\mathcal H_{\mathrm{glob}}$ is invariant under $M$, unit-synchronization is preserved under the dynamics generated by any star-dependent matrix. That is for unit synchronization, units are synchronization preserving clusters. 
	The following corollary describes the asymptotic behavior of unit-synchronization.
	
	\begin{corollary}[Unit-synchronization structure]
		\label{cor:unit-synchronization}
		Let $M_n$ be a star-dependent matrix associated with
		$H_n \sim \mathcal H(n,k,p_n)$.
		
		\begin{enumerate}[label=(\roman*), leftmargin=2em]
			
			\item If $\lambda_n=\log n+c$, then asymptotically all unit-synchronized states are trivial.
			
			\item If
			$\lambda_n=\tfrac{1}{2}(\log n+\log\log n)+w$,
			then unit-synchronized states are supported on a finite random collection of clusters.
			
		\end{enumerate}
	\end{corollary}
	
	We emphasize that unit-synchronization captures synchronization induced by star-equivalence and does not preclude additional synchronization patterns arising from the quotient dynamics.
	\subsubsection{Reconstruction and asymptotic vertex fingerprinting}
	\label{sec:reconstruction}
	
	Recall that vertices of a hypergraph are indistinguishable with respect to all star-dependent observables if and only if they belong to the same unit. Equivalently, the unit partition is the coarsest partition of $V(H_n)$ that preserves all star-based information.
	
	\begin{proposition}[Asymptotic vertex fingerprinting by stars]
		\label{thm:asymptotic_fingerprinting}
		Let $H_n \sim \mathcal H(n,k,p_n)$. If
		$\lambda_n = \log n + c$
		for some fixed $c \in \mathbb R$, then with high probability
		distinct vertices have distinct stars.
	\end{proposition}
	Equivalently, with high probability every unit is trivial, so the natural projection
	$
	V(H_n)\to\mathfrak U(H_n)
	$
	is bijective.	
	
	\section{Proof of the results}\label{sec:proof}
	\label{sec:proofs}
	
	We organize the proofs according to the two regimes identified in the 
	previous section. In the regime $\lambda_n = \log n + O(1)$, the arguments 
	are driven by the behavior of isolated vertices, which determine the 
	degenerate star-collision structure. In particular, the proof of 
	Theorem~\ref{thm:degenerate-count} relies on a factorial moment computation, 
	reducing the problem to the Poisson limit for isolated vertices.
	
	In contrast, the analysis in the regime 
	$\lambda_n = \tfrac{1}{2}(\log n + \log\log n) + O(1)$ requires a more 
	refined counting of small configurations, leading to Poisson limits for 
	non-degenerate star-collisions.
	
	Throughout the proofs, we use standard asymptotic notation.
	For sequences $a_n,b_n$, we write $a_n = O(b_n)$ if $|a_n| \le C b_n$
	for some constant $C>0$ independent of $n$, and $a_n = o(b_n)$ if
	$a_n/b_n \to 0$. We write $a_n = \Theta(b_n)$ if $a_n = O(b_n)$ and
	$b_n = O(a_n)$. All implicit constants may depend on $k$ and on any
	fixed parameter such as $t$. All asymptotic statements and limits are understood as $n\to\infty$,
	unless explicitly stated otherwise.
	
	We repeatedly use Markov's inequality in the form
	$\mathbb{P}(X_n \ge 1) \le \mathbb{E}[X_n]$,
	so that $\mathbb{E}[X_n] \to 0$ implies $\mathbb{P}(X_n \ge 1) \to 0$.
	All probabilities are taken with respect to the random hypergraph
	$\mathcal{H}(n,k,p_n)$, in which hyperedges are included independently.
	
	We now proceed to the proofs of the results in each regime.
	\subsection{Proofs of star-collision results}We collect here the proofs of the results stated in 
	Subsection~\ref{sec:main result-1}. The analysis is based on a common 
	combinatorial expression for the expectation of the number 
	of star-collisions of a given support size, which we derive 
	first. This expression serves as the starting point for all 
	subsequent arguments.
	
	The results in the regime $\lambda_n = \log n + O(1)$ are 
	obtained by analyzing this expectation together with standard 
	first and second moment arguments. In particular, the behavior 
	of degenerate star-collisions is closely tied to the number 
	of isolated vertices.
	
	In the sparser regime 
	$\lambda_n = \tfrac{1}{2}(\log n + \log\log n) + w$, 
	the same expectation estimate, combined with factorial moment 
	calculations, yields Poisson limit laws and structural results 
	for non-degenerate star-collisions and units.
	Fix a pair of distinct vertices $\{u,v\} \subset [n]$.
	Partition the family of $k$-subsets of $[n]$ according to their 
	intersection with $\{u,v\}$. There are
	$\binom{n-2}{k-2}$ subsets containing both $u$ and $v$, and 
	$\binom{n-2}{k-1}$ subsets containing $u$ but not $v$ (and the same 
	number containing $v$ but not $u$).
	
	For $\Gamma(u)=\Gamma(v)$ with $|\Gamma(u)|=r$, exactly $r$ of the 
	$\binom{n-2}{k-2}$ subsets containing both vertices must appear as 
	hyperedges, while the remaining such subsets must be absent. Moreover, 
	all subsets containing exactly one of $u$ and $v$ must be absent.
	
	By independence of edges in $\mathcal{H}(n,k,p_n)$,
	\begin{align}\label{eq:pair-collision}
		&\mathbb{P}(\text{$\{u,v\}$ forms a star-collision of support size $r$ })
		\\\notag&=
		\binom{\binom{n-2}{k-2}}{r}
		p_n^r
		(1-p_n)^{2\binom{n-2}{k-1} + \binom{n-2}{k-2}-r}.
	\end{align}
	This equation leads us to the following Lemma.
	\begin{lemma}[Expected number of star-collisions]
		\label{lem:EXr}
		Let $k \ge 2$ be fixed and let $H_n \in \mathcal{H}(n,k,p_n)$.
		For each $r \ge 0$, let $X_r$ denote the number of star-collisions
		of support size $r$. Then
		\begin{equation}\label{eq:EXr}
			\mathbb{E}[X_r]
			=
			\binom{n}{2}
			\binom{\binom{n-2}{k-2}}{r}
			p_n^r
			(1-p_n)^{2\binom{n-2}{k-1} + \binom{n-2}{k-2} - r}.
		\end{equation}
		In terms of $\lambda_n = \binom{n-1}{k-1} p_n$, if $\lambda_n=o(n^{(k-1)/2})$ and $\lambda_n=o(n)$, then for each fixed $r$,
		\begin{equation}\label{eq:EXr_asymp}
			\mathbb{E}[X_r]
			=
			\frac{(k-1)^r}{2\, r!}
			\, n^{2-r}
			\, \lambda_n^r
			\, \exp(-2\lambda_n)
			\,(1+o(1)).
		\end{equation}
	\end{lemma}
	\begin{proof}
		By summing \eqref{eq:pair-collision} over all $\binom{n}{2}$ pairs, we obtain \eqref{eq:EXr}.
		For the asymptotic form, note that
		$p_n = \lambda_n / \binom{n-1}{k-1}$ and
		$\binom{n-2}{k-2}
		=
		\left(1+o(1)\right)\frac{k-1}{n}\binom{n-1}{k-1}$.
		For fixed $r$, we have
		$
	\binom{\binom{n-2}{k-2}}{r}
	=
	\frac{1+o(1)}{r!}
	\left(\binom{n-2}{k-2}\right)^r.
		$
		Moreover, since $p_n \to 0$, we have
		$\log(1-p_n) = -p_n + O(p_n^2).$
		Thus,
		\begin{align*}
			&(1-p_n)^{2\binom{n-2}{k-1} + \binom{n-2}{k-2} - r} \\
			&\quad =
			\exp\!\Big(
			-\;p_n\big(2\binom{n-2}{k-1} + \binom{n-2}{k-2} - r\big)
			\;+\;
			O\!\big(p_n^2 \binom{n-1}{k-1}\big)
			\Big).
		\end{align*}
		Using
		$p_n \binom{n-1}{k-1} = \lambda_n
		\quad \text{and} \quad
		2\binom{n-2}{k-1} + \binom{n-2}{k-2}
		= 2\binom{n-1}{k-1} - \binom{n-2}{k-2},$
		we obtain $$p_n\big(2\binom{n-2}{k-1} + \binom{n-2}{k-2} - r\big)
		=
		2\lambda_n
		-
		p_n\binom{n-2}{k-2}
		-
		rp_n.$$ Using $\lambda_n=o(n)$ we have $p_n\binom{n-2}{k-2}
		=
		O(\lambda_n/n)
		=o(1)$. Thus, we have
		\[
		p_n\big(2\binom{n-2}{k-1} + \binom{n-2}{k-2} - r\big)
		= 2\lambda_n + o(1),
		\]
		while using $\lambda_n=o(n^{(k-1)/2})$, 
		$p_n^2 \binom{n-1}{k-1} \to 0$.
		Therefore,
		$(1-p_n)^{2\binom{n-2}{k-1} + \binom{n-2}{k-2} - r}
		=
		\exp(-2\lambda_n + o(1))$. Substituting these estimates into \eqref{eq:EXr} yields
		\[
		\mathbb{E}[X_r]
		=
		\frac{(k-1)^r}{2r!}
		n^{2-r}\lambda_n^r
		\exp(-2\lambda_n)
		(1+o(1)).
		\]
	\end{proof}
	The asymptotic assumptions above are satisfied in all regimes considered later in this paper, including the finite-$\lambda$ regime, the logarithmic regime $\lambda_n=\log n + c$, and the critical logarithmic regime
	$\lambda_n=\tfrac12(\log n+\log\log n)+c$.
	
	\subsubsection{Proof of Theorem~\ref{thm:star-collision-critical}}
	\label{proof:thm:star-collision-critical}
	We prove Theorem~\ref{thm:star-collision-critical}, which describes the behavior of $X_r$ in the regime $\lambda_n = \log n + c$.
	\begin{proof}
		We use Lemma~\ref{lem:EXr} under the scaling $\lambda_n = \log n + c$.
		
		\medskip
		\noindent
		\textbf{(i) The case $r \ge 1$.}
		From Lemma~\ref{lem:EXr}, for each fixed $r \ge 1$,
		\[
		\mathbb{E}[X_r]
		=
		\frac{(k-1)^r}{2\,r!}
		\, n^{2-r}
		\, \lambda_n^r
		\, e^{-2\lambda_n}
		\,(1+o(1)).
		\]
		Substituting $\lambda_n = \log n + c$, we obtain
		\[
		\mathbb{E}[X_r]
		=
		\frac{(k-1)^r}{2\,r!}
		(\log n + c)^r
		\, n^{-r}
		\, e^{-2c}
		\,(1+o(1)).
		\]
		Since $(\log n)^r = o(n^r)$, it follows that
		$\mathbb{E}[X_r] \longrightarrow 0$ as $n\to\infty$.
		By Markov's inequality,
		$\mathbb{P}(X_r \ge 1)
		\le
		\mathbb{E}[X_r]
		\longrightarrow 0$.

		\medskip
		\noindent
		\textbf{(ii) The case $r=0$.}
		From Lemma~\ref{lem:EXr}, we have
		\[
		\mathbb{E}[X_0]
		=
		\binom{n}{2}
		e^{-2\lambda_n}
		(1+o(1)).
		\]
		Substituting $\lambda_n = \log n + c$, we obtain
		\[
		\mathbb{E}[X_0]
		=
		\frac{n^2}{2}
		\cdot n^{-2}
		e^{-2c}
		(1+o(1))
		\longrightarrow
		\frac{1}{2} e^{-2c}.
		\]
	\end{proof}
	The vanishing of all non-degenerate star-collisions has a direct
	structural consequence for the unit decomposition of the hypergraph.
\begin{proof}[Proof of Corollary~\ref{cor:units-trivial}]
	A non-trivial unit exists if and only if there exists a pair of
	distinct vertices with identical non-empty stars. Equivalently,
	\[
	\{\exists\, W \in \mathfrak U(H_n)\text{ with } |W|\ge2\}
	=
	\left\{\sum_{r\ge1} X_r \ge 1\right\}.
	\]
	Therefore, by Markov's inequality,
	\[
	\mathbb P\!\left(\exists\, W \in \mathfrak U(H_n)\text{ with } |W|\ge2\right)
	=
	\mathbb P\!\left(\sum_{r\ge1}X_r\ge1\right)
	\le
	\sum_{r\ge1}\mathbb E[X_r].
	\]
	
	Using the exact formula from Lemma~\ref{lem:EXr},
\begin{align*}
	\sum_{r\ge1}\mathbb E[X_r]
	&=
	\binom n2
	\sum_{r\ge1}
	\binom{\binom{n-2}{k-2}}r
	p_n^r
	(1-p_n)^{2\binom{n-2}{k-1}+\binom{n-2}{k-2}-r}
	\\
	&=
	\binom n2
	(1-p_n)^{2\binom{n-2}{k-1}+\binom{n-2}{k-2}}
	\sum_{r\ge1}
	\binom{\binom{n-2}{k-2}}r
	\left(\frac{p_n}{1-p_n}\right)^r
	\\
	&=
	\binom n2
	(1-p_n)^{2\binom{n-2}{k-1}}
	\left(
	1-(1-p_n)^{\binom{n-2}{k-2}}
	\right),
\end{align*}
	where the last step follows from the binomial theorem.
	
	Now,
	$
	(1-p_n)^{2\binom{n-2}{k-1}}
	=
	n^{-2}e^{-2c}(1+o(1)),
$	and
	\[
	1-(1-p_n)^{\binom{n-2}{k-2}}
	=
	O\!\left(
	p_n\binom{n-2}{k-2}
	\right)
	=
	O\!\left(\frac{\log n}{n}\right).
	\]
That is,
	$
	\sum_{r\ge1}\mathbb E[X_r]
	=
	O\!\left(\frac{\log n}{n}\right)
	\longrightarrow 0.
	$
	Therefore,
	\[
	\mathbb P(\exists\, W \in \mathfrak U(H_n)\text{ with } |W|\ge2)
	\longrightarrow 0.
	\]
\end{proof}
	In view of Theorem~\ref{thm:star-collision-critical} and  Corollary~\ref{cor:units-trivial}, the star-collision structure
	in the regime $\lambda_n = \log n + c$ is entirely governed by
	degenerate collisions. It therefore remains to understand the behavior
	of the number of degenerate star-collisions
	$ X_0$.
	We begin by identifying the threshold for the existence of such
	collisions.
	\subsubsection{Proof of Theorem~\ref{thm:degenerate-threshold}}
	\begin{proof}
		Let $I_n$ denote the number of isolated vertices in 
		$H_n \in \mathcal{H}(n,k,p_n)$. By definition, a vertex is isolated 
		if none of the $\binom{n-1}{k-1}$ edges incident to it are present. 
		Thus,
		$\mathbb{P}(\text{a fixed vertex is isolated})
		=
		(1-p_n)^{\binom{n-1}{k-1}}.$
		By linearity of expectation,
		\[
		\mathbb{E}[I_n]
		=
		n (1-p_n)^{\binom{n-1}{k-1}}
		=
		n \exp(-\lambda_n + o(1)).
		\]
		We now distinguish the two regimes.
		
		\medskip
		\noindent
		\textbf{(i) Supercritical case: $\lambda_n = \log n + w(n)$, $w(n)\to+\infty$.}
		In this case,
		$$
		\mathbb{E}[I_n]
		=
		n e^{-\lambda_n + o(1)}
		=
		e^{-w(n) + o(1)}
		\longrightarrow 0.
		$$
		Thus, by Markov's inequality,
		$\mathbb{P}(I_n \ge 1)
		\le
		\mathbb{E}[I_n]
		\longrightarrow 0$,
		so $\mathbb{P}(I_n = 0) \to 1$.
		
		\medskip
		\noindent
		\textbf{(ii) Subcritical case: $\lambda_n = \log n - w(n)$, $w(n)\to+\infty$.}
		To control fluctuations, write $I_n = \sum_{u} J_u$, where 
		$J_u = \mathbf{1}_{\{u \text{ is isolated}\}}$. Then
		\[
		\mathrm{Var}(I_n)
		=
		\sum_u \mathrm{Var}(J_u)
		+
		\sum_{u\ne v} \mathrm{Cov}(J_u,J_v).
		\]
		Since $\mathrm{Var}(J_u) \le \mathbb{E}[J_u]$, we have
		$\sum_u \mathrm{Var}(J_u) \le \mathbb{E}[I_n].$
		For $u \ne v$, we have
		\[
		\mathrm{Cov}(J_u,J_v)
		=
		\mathbb{E}[J_uJ_v] - \mathbb{E}[J_u]\mathbb{E}[J_v]
		=
		(1-p_n)^{2\binom{n-1}{k-1}}
		\left[
		(1-p_n)^{-\binom{n-2}{k-2}} - 1
		\right].
		\]
		Using the expansion $\log(1-x) = -x + O(x^2)$ as $x \to 0$, we obtain
		\[
		(1-p_n)^{-\binom{n-2}{k-2}}
		=
		\exp\!\left(
		p_n \binom{n-2}{k-2}
		+
		O\!\left(p_n^2 \binom{n-2}{k-2}\right)
		\right).
		\]
		Since $p_n \binom{n-1}{k-1} = \lambda_n$ and 
		$\binom{n-2}{k-2}/\binom{n-1}{k-1} = \frac{k-1}{n-1}$, it follows that
		\[
		p_n \binom{n-2}{k-2}
		=
		\frac{k-1}{n-1}\,\lambda_n
		=
		O\!\left(\frac{\lambda_n}{n}\right),
		\]
		and similarly the quadratic term is negligible. Since $\exp(o(1)) = 1 + o(1)$, it follows that
		$(1-p_n)^{-\binom{n-2}{k-2}} - 1
		=
		O\!\left(\frac{\lambda_n}{n}\right)=o(1)$. Since in the present regime
		$
		\lambda_n=\log n-w(n)\le \log n,
		$
		we have $\lambda_n=o(n)$.
		Thus,
		\[
		\mathrm{Cov}(J_u,J_v)
		=
		\mathbb{E}[J_u]\mathbb{E}[J_v]
		\cdot O\!\left(\frac{\lambda_n}{n}\right)=
		\mathbb{E}[J_u]\mathbb{E}[J_v]
		\cdot o(1),
		\]
		and summing over all $u\ne v$ yields
		$
		\sum_{u\ne v} \mathrm{Cov}(J_u,J_v)
		=
		o\!\left(\mathbb{E}[I_n]^2\right).
		$ Therefore,
		\[
		\mathrm{Var}(I_n)
		=
		O(\mathbb{E}[I_n])
		+
		o(\mathbb{E}[I_n]^2).
		\]
		In this case,
		$$
		\mathbb{E}[I_n]
		=
		n e^{-\lambda_n + o(1)}
		=
		e^{w(n) + o(1)}
		\longrightarrow \infty.
		$$
		From the variance bound
		$\mathrm{Var}(I_n)
		=
		O(\mathbb{E}[I_n])
		+
		o(\mathbb{E}[I_n]^2)$, it follows that
		$\frac{\mathrm{Var}(I_n)}{(\mathbb{E}[I_n])^2}
		\to 0,$
		since $\mathbb{E}[I_n] \to \infty$ in this regime.
		Now, we apply Chebyshev's inequality to obtain
		\[
		\mathbb{P}\bigl(|I_n - \mathbb{E}[I_n]| \ge \tfrac{1}{2}\mathbb{E}[I_n]\bigr)
		\;\le\;
		\frac{4\,\mathrm{Var}(I_n)}{\mathbb{E}[I_n]^2}
		\;\longrightarrow\; 0.
		\]
		That is,
		$\mathbb{P}\!\left(I_n \ge \tfrac{1}{2}\mathbb{E}[I_n]\right)
		\longrightarrow 1.$ Since $\mathbb{E}[I_n] \to \infty$, we have
		$\tfrac{1}{2}\mathbb{E}[I_n] \ge 2$ for all sufficiently large $n$.
		Thus,
		$\{ I_n \ge \tfrac{1}{2}\mathbb{E}[I_n] \}
		\subseteq
		\{ I_n \ge 2 \}$ for sufficiently large $n$.
		That is,
		$\mathbb{P}(I_n \ge 2)
		\;\ge\;
		\mathbb{P}\!\left(I_n \ge \tfrac{1}{2}\mathbb{E}[I_n]\right)
		\longrightarrow 1$.
			Finally, observe that
		$X_0 = \binom{I_n}{2}$.
		Thus, we have $X_0 > 0$ if and only if $I_n \ge 2$.
		Therefore, the two statements of the theorem now follow immediately.
	\end{proof}
	The above estimates for $I_n$ are standard and follow from classical
	results on isolated vertices in random hypergraphs; see \cite{burghart2024hitting,frieze2025threshold}.
	Both Theorem~\ref{thm:degenerate-threshold} and
	Theorem~\ref{thm:degenerate-count} are governed by the behavior of
	isolated vertices. Indeed, recalling that
	$X_0 = \binom{I_n}{2}$,
	where $I_n$ denotes the number of isolated vertices, the threshold
	behavior of $X_0$ follows from the growth of $I_n$, while a more
	refined analysis of $I_n$ yields the limiting distribution of $X_0$.
	We now determine this limiting distribution in the critical window.
	\subsubsection{Degenerate star-collisions in the critical window: Proof of Theorem~\ref{thm:degenerate-count}}
	\label{proof:thm:degenerate-count}
	
	\begin{proof}
		We prove Theorem~\ref{thm:degenerate-count} by relating degenerate
		star-collisions to isolated vertices.
		Let $I_n$ denote the number of isolated vertices in
		$H_n \in \mathcal{H}(n,k,p_n)$, that is,
		\[
		I_n = \sum_{v \in [n]} \mathbf{1}_{\{v \text{ is isolated}\}}.
		\]
		By definition, a degenerate star-collision corresponds to a pair of
		vertices with empty stars, and thus
		$X_0 = \binom{I_n}{2}$.
		For a fixed vertex $v$, the number of edges containing $v$ is
		$\binom{n-1}{k-1}$. Thus
		\[
		\mathbb{P}(v \text{ is isolated})
		=
		(1 - p_n)^{\binom{n-1}{k-1}}.
		\]
		Using $p_n = \lambda_n / \binom{n-1}{k-1}$ with
		$\lambda_n = \log n + c$, and the expansion
		$\log(1 - p_n) = -p_n + O(p_n^2)$, we obtain
		$(1 - p_n)^{\binom{n-1}{k-1}}
		=
		\exp(-\lambda_n + o(1))
		=
		n^{-1} e^{-c}(1 + o(1))$.
		Therefore,
		$$
		\mathbb{E}[I_n]
		=
		n \cdot n^{-1} e^{-c}(1 + o(1))
		\longrightarrow
		e^{-c}.
		$$
		To identify the limiting distribution, we compute factorial moments(\cite{Bollobas-random-graphs}[Section~1.5]).
		For fixed $r \ge 1$,
		\[
		\mathbb{E}[(I_n)_r]
		=
		\sum_{(v_1,\dots,v_r)\ \text{distinct}}
		\mathbb{P}(v_1,\dots,v_r \text{ are all isolated}).
		\]
		Let $S = \{v_1,\dots,v_r\}$. The vertices in $S$ are isolated if and only if no edge intersects $S$. The number of such forbidden edges is
		$F_S = \binom{n}{k} - \binom{n-r}{k}$.
		A standard expansion gives
		$F_S
		=
		r\binom{n-1}{k-1}
		-
		\binom{r}{2}\binom{n-2}{k-2}
		+
		O(n^{k-3})$.
		Using $p_n \binom{n-1}{k-1} = \lambda_n$ and
		$\binom{n-2}{k-2}/\binom{n-1}{k-1} = O(1/n)$, we obtain
		$p_n F_S = r\lambda_n + o(1)$.
		
		Thus,
		$\mathbb{P}(S \text{ is isolated})
		=
		(1 - p_n)^{F_S}
		=
		\exp(-r\lambda_n + o(1))
		=
		n^{-r} e^{-rc}(1 + o(1)).$
		Therefore,
		$\mathbb{E}[(I_n)_r]
		=
		r!\binom{n}{r} n^{-r} e^{-rc}(1 + o(1))
		\longrightarrow
	(e^{-c})^r$.
		Thus, the factorial moments of $I_n$ converge to those of a
		$\mathrm{Poisson}(e^{-c})$ random variable, and therefore,
		\[
		I_n \xrightarrow{d} \mathrm{Poisson}(e^{-c}).
		\]
		
		\medskip
		
		\noindent
		Since $X_0 = \binom{I_n}{2}$ and the function
		$
		f(x) = \binom{x}{2} = \frac{x(x-1)}{2}
		$
		is continuous on $\mathbb{R}$, it follows from the
		continuous mapping theorem that
		\[
	X_0\xrightarrow{d} \binom{Z}{2},
		\quad
		\text{where } Z \sim \mathrm{Poisson}(e^{-c}).
		\]
		 This completes the proof.
	\end{proof}
	We now move to a sparser regime in which non-degenerate star-collisions
	appear with non-trivial probability. At the scale
	$
	\lambda_n
	=
	\tfrac{1}{2}(\log n + \log\log n) + w,
	$
	the structure of star-collisions undergoes a qualitative change.
	In the previous regime, all collisions were induced by isolated vertices,
	so that the analysis reduced to properties of the isolated vertex count.
	In contrast, star-collisions of support size $1$ now arise from genuine
	edge configurations and can no longer be reduced to vertex-level events.
	As a result, one must directly control the interaction between multiple
	such configurations. This leads to a Poisson limit for $X_1$, which we
	establish via the method of factorial moments.
	\subsubsection{Proof of Theorem~\ref{thm:r1-poisson}}
	\begin{proof}
		Let
		$
		X_1 = \sum_{u<v} I_{uv},
		$
		where $I_{uv}$ indicates that $\{u,v\}$ forms a $1$–support star collision.
		From Lemma~\ref{lem:EXr}, under
		$
		\lambda_n = \tfrac12(\log n + \log\log n) + w,
		$
		we have
		$$
		\mathbb{E}[X_1] \to \mu(w) := \frac{k-1}{4} e^{-2w}.
		$$
			We prove convergence in distribution via factorial moments.
		Fix $t \ge 1$:
		\[
		\mathbb{E}[(X_1)_t]
		=
		\sum_{(\{u_1,v_1\},\dots,\{u_t,v_t\})}
		\mathbb{P}\!\left(\prod_{i=1}^t I_{u_i v_i}=1\right),
		\]
		where the sum is over ordered $t$-tuples of distinct unordered pairs.
		Let
		$
		S = \bigcup_{i=1}^t \{u_i,v_i\}, \quad s = |S|.
		$
		
		\medskip
		\noindent
		\emph{Overlapping configurations ($s<2t$).}
		The number of such configurations is at most $C_t n^s$.
		Fix one such configuration.
		
		For each $i$, define
		$\mathcal C_{u_i v_i}
		=
		\{e : \{u_i,v_i\} \subset e\},
		\qquad
		\mathcal S_{u_i v_i}
		=
		\{e : |e \cap \{u_i,v_i\}| = 1\}.$
			Then $\prod_{i=1}^t I_{u_i v_i}=1$ occurs iff there exist exactly one edge
		$e_i \in \mathcal C_{u_i v_i}$ such that all $\{e_i:i=1,\ldots,t\}$ are present and all edges in
		$\mathcal F(e_1,\dots,e_t)
		:=
		\bigcup_{i=1}^t
		\Big(
		(\mathcal C_{u_i v_i} \setminus \{e_i\})
		\cup
		\mathcal S_{u_i v_i}
		\Big)$
		are absent.
		By independence,
		\[
		\mathbb P\!\left(\prod_{i=1}^t I_{u_i v_i}=1\right)
		=
		\sum_{e_1,\dots,e_t}
		p_n^{\,m(e_1,\dots,e_t)}
		(1-p_n)^{|\mathcal F(e_1,\dots,e_t)|},
		\]
		where $m(e_1,\dots,e_t) \le t$ is the number of distinct edges. We decompose the above sum according to the number
		$m=m(e_1,\dots,e_t)\in\{1,\dots,t\}$ of distinct witness edges.
		
		The case $m=t$ corresponds to all selected edges being distinct and will give
		the principal contribution. The cases $m<t$ are negligible, since any repeated
		edge must serve at least two pairs and therefore satisfies additional vertex
		constraints, reducing the number of admissible choices by a factor of order
		$n^{-2}$; the corresponding counting argument is identical to the one carried out below in the disjoint case.
		
		Intersections of two $\mathcal C_{u_i v_i}$-families are $O(n^{k-3})$,
		whereas intersections involving at least one $\mathcal S_{u_i v_i}$-family are $O(n^{k-2})$.
			Using inclusion–exclusion and the fact that all pairwise intersections
		of the families $\mathcal C_{u_i v_i}, \mathcal S_{u_i v_i}$ have size at most
		$O(n^{k-2})$, we obtain uniformly
		\[
		|\mathcal F(e_1,\dots,e_t)|
		=
		t\Big(\binom{n-2}{k-2} - 1 + 2\binom{n-2}{k-1}\Big)
		+
		O(t^2 n^{k-2}),
		\]
		and therefore,
		\[
		(1-p_n)^{|\mathcal F|}
		=
		\exp(-2t\lambda_n + o(1)).
		\]
			Since $|\mathcal C_{u_i v_i}| = \Theta(n^{k-2})$ and $e_i\in \mathcal C_{u_i v_i}$, it follows that
		\[
		\mathbb P\!\left(\prod_{i=1}^t I_{u_i v_i}=1\right)
		=
		O\!\left(
		n^{t(k-2)} p_n^t e^{-2t\lambda_n}
		\right)
		=
		O(n^{-2t}).
		\]
			Thus the total contribution of overlapping configurations is
		$
		O(n^s n^{-2t}) = o(1),
		$
		since $s<2t$.
		
		\medskip
		\noindent
		\emph{Disjoint configurations ($s=2t$).}
		Assume the pairs are vertex–disjoint and write $\mathcal C_i,\mathcal S_i$
		for the corresponding edge families.
		Arguing as above, for any choice of edges $e_i \in \mathcal C_i$,
		\[
		(1-p_n)^{|\mathcal F(e_1,\dots,e_t)|}
		=
		\exp(-2t\lambda_n + o(1)).
		\]
		Thus,
		\[
		\mathbb P\!\left(\prod_{i=1}^t I_{u_i v_i}=1\right)
		=
		\sum_{e_1,\dots,e_t}
		p_n^{\,m(e_1,\dots,e_t)}
		\exp(-2t\lambda_n + o(1)).
		\]
		We now separate configurations according to whether the edges
		$e_1,\dots,e_t$ are distinct.
			If all $e_i$ are distinct, then $m=t$, and the number of such choices is
		$(\binom{n-2}{k-2})^t$. Thus their contribution is
		\[
		(\binom{n-2}{k-2} p_n e^{-2\lambda_n})^t (1+o(1)).
		\]
		
		
		If some edges coincide, suppose exactly $m<t$ distinct edges are used.
		Then at least one edge contains vertices from two distinct pairs, and
		hence can be chosen in at most $\binom{n-4}{k-4}=O(n^{k-4})$ ways,
		while edges used by a single pair contribute $O(n^{k-2})$ choices.
		It follows that the number of such configurations is at most
		\[
		O\!\left(n^{m(k-2)-2}\right).
		\]
			For each such configuration, the probability contribution is at most
		$p_n^m e^{-2t\lambda_n}$. Using
		$
		p_n = \Theta\!\left(\frac{\log n}{n^{k-1}}\right)
		$
		and
		$
		e^{-2t\lambda_n}
		=
		\Theta\!\left(\frac{1}{n^t (\log n)^t}\right),
		$
		we obtain
		\[
		p_n^m e^{-2t\lambda_n}
		=
		O\!\left(
		\frac{(\log n)^m}{n^{m(k-1)}} \cdot \frac{1}{n^t (\log n)^t}
		\right)
		=
		O\!\left(
		n^{-m(k-1)-t} (\log n)^{m-t}
		\right).
		\]
		Multiplying with the number of configurations yields
		\[
		O\!\left(
		n^{m(k-2)-2}
		\cdot
		n^{-m(k-1)-t}
		(\log n)^{m-t}
		\right)
		=
		O\!\left(
		n^{-m - t - 2}
		(\log n)^{m-t}
		\right).
		\]
		Since $m \le t-1$, we have $m+t+2 \ge 2t+1$, and hence
		\[
		n^{-m - t - 2}
		(\log n)^{m-t}
		=
		o(n^{-2t}).
		\]
		Thus configurations with $m<t$ contribute negligibly.
		Therefore only configurations in which the edges $e_1,\dots,e_t$ are
		distinct contribute to the leading order. In this case, summing over
		$e_i \in \mathcal C_i$ yields
		\[
		\mathbb P\!\left(\prod_{i=1}^t I_{u_i v_i}=1\right)
		=
		\left(\binom{n-2}{k-2} p_n e^{-2\lambda_n}\right)^t (1+o(1)).
		\]
		Comparing with the single-pair estimate from \eqref{eq:pair-collision}
		(with $r=1$), we obtain
		\[
		\mathbb P\!\left(\prod_{i=1}^t I_{u_i v_i}=1\right)
		=
		\prod_{i=1}^t \mathbb P(I_{u_i v_i}=1)\,(1+o(1)).
		\]
		The number of ordered $t$-tuples of pairwise disjoint vertex pairs is
		\[
		(1+o(1))\frac{n^{2t}}{2^t}.
		\]
		so
		\[
		\sum_{\text{disjoint}}
		\mathbb P\!\left(\prod I_{u_i v_i}=1\right)
		\to \mu(w)^t.
		\]
		
		\medskip
		\noindent
		Combining both cases,
		$
		\mathbb{E}[(X_1)_t] \to \mu(w)^t,
		$
		and therefore, $X_1 \xrightarrow{d} \mathrm{Poisson}(\mu(w))$.
	\end{proof}
	We now turn to higher-order star-collisions.
	In contrast to the $r=1$ case, the behavior depends sharply on the
	support size.
	
	For $r \ge 3$, the expected number of such collisions tends to zero,
	and a direct application of Markov's inequality shows that they do not
	occur with high probability.
	
	The case $r=2$ is critical: under the scaling $\lambda_n \to \lambda \in (0,\infty)$,
	the expected number of collisions converges to a positive constant.
	We show that in this regime $X_2$ converges in distribution to a Poisson
	random variable via the method of factorial moments.
	\subsubsection{Proof of Proposition~\ref{prop:high-order-collisions}}
	\begin{proof}
		We treat the two parts separately.
		
		\medskip
		\noindent
		\emph{(ii) Case $r \ge 3$.}
		From Lemma~\ref{lem:EXr},We recall that
		\[
		\mathbb E[X_r]
		=
		\frac{(k-1)^r}{2r!}\,
		n^{2-r}\,
		\lambda_n^r
		e^{-2\lambda_n}(1+o(1)).
		\]
		Since $r \ge 3$, we have $n^{2-r} \to 0$.
		Moreover, the function $x \mapsto x^r e^{-2x}$ is bounded on $[0,\infty)$ and tends to $0$ as $x\to\infty$.
		Hence $\lambda_n^r e^{-2\lambda_n} = O(1)$ for any sequence $(\lambda_n)$.
		Therefore,
		\[
		\mathbb E[X_r] = O(n^{2-r}) \longrightarrow 0.
		\]
		By Markov's inequality,
		\[
		\mathbb{P}(X_r \ge 1)
		\le
		\mathbb{E}[X_r]
		\longrightarrow 0.
		\]
		\medskip
		\noindent
		\emph{(i) Case $r=2$.}
		Let
		$
		X_2 = \sum_{u<v} I_{uv},
		$
		where $I_{uv}$ indicates that $\{u,v\}$ forms a star-collision of support size $2$.
		From Lemma~\ref{lem:EXr}, if $\lambda_n \to \lambda \in (0,\infty)$, then
		\[
		\mathbb{E}[X_2]
		\longrightarrow
		\mu
		:=
		\frac{(k-1)^2}{4}\lambda^2 e^{-2\lambda}.
		\]
		We prove convergence in distribution via factorial moments.
		Fix $t \ge 1$:
		\[
		\mathbb{E}[(X_2)_t]
		=
		\sum_{(\{u_1,v_1\},\dots,\{u_t,v_t\})}
		\mathbb{P}\!\left(
		\prod_{i=1}^t I_{u_i v_i}=1
		\right),
		\]
		where the sum is over ordered $t$-tuples of distinct unordered vertex pairs.
			Let
		$$
		S = \bigcup_{i=1}^t \{u_i,v_i\},
		\;
		s = |S|.
		$$
		
		\medskip
		\noindent
		\emph{Overlapping configurations ($s<2t$).}
		The number of such configurations is at most $C_t n^s$.
		Fix one such configuration.
		
		Arguing as in the proof of Theorem~\ref{thm:r1-poisson},
		the joint event can be realized using at most $2t$ witness edges,
		with each witness edge chosen from a family of size
		$O(n^{k-2})$ and the absence of
		a forbidden family $\mathcal F$ satisfying
		$p_n |\mathcal F|
		=
		2t\lambda_n + o(1)$.	Hence
		\[
		\mathbb{P}\!\left(\prod_{i=1}^t I_{u_i v_i}=1\right)
		=
		O\!\left(
		n^{2t(k-2)} p_n^{2t}
		\right).
		\]
		Using $p_n \sim \lambda / n^{k-1}$, we obtain
		\[
		\mathbb{P}\!\left(\prod_{i=1}^t I_{u_i v_i}=1\right)
		=
		O(n^{-2t}).
		\]
		Thus the total contribution is
		\[
		O(n^s n^{-2t}) = o(1),
		\qquad (s<2t).
		\]
		
		\medskip
		\noindent
		\emph{Disjoint configurations ($s=2t$).}
		Assume that the pairs are vertex–disjoint.
		For each $i$, choose two edges $e_{i,1}, e_{i,2} \in \mathcal C_{u_i v_i}$.
		Arguing as above and using inclusion–exclusion, uniformly over all choices,
		\[
		p_n |\mathcal F|
		=
		2t\lambda_n + o(1),
		\qquad
		(1-p_n)^{|\mathcal F|}
		=
		e^{-2t\lambda_n + o(1)}.
		\]
		
		Configurations in which some selected witness edges coincide contribute
		only lower-order terms, since an edge containing vertices from two
		distinct pairs admits only $O(n^{k-4})$ choices.
		Summing over all choices of edges yields
		\[
		\mathbb{P}\!\left(
		\prod_{i=1}^t I_{u_i v_i}=1
		\right)
		=
		\left(\binom{n-2}{k-2}^2 p_n^2 e^{-2\lambda_n}\right)^t
		(1+o(1)).
		\]
		By the single-pair estimate, this equals
		\[
		\prod_{i=1}^t \mathbb{P}(I_{u_i v_i}=1)\,(1+o(1)).
		\]
		The number of ordered disjoint $t$-tuples is
		$
		(1+o(1))\frac{n^{2t}}{2^t},
		$
		and therefore
		\[
		\sum_{\text{disjoint}}
		\mathbb{P}\!\left(\prod I_{u_i v_i}=1\right)
		=
		(\mathbb{E}[X_2])^t + o(1).
		\]
		
		\medskip
		\noindent
		Combining both cases,
		$
		\mathbb{E}[(X_2)_t] \to \mu^t.
		$
		Hence all factorial moments converge to those of a Poisson$(\mu)$
		random variable, and therefore
		\[
		X_2 \xrightarrow{d} \mathrm{Poisson}(\mu).
		\]
		This completes the proof.
	\end{proof}
	We now turn to the global structure of non-trivial units.
	While Theorem~\ref{thm:r1-poisson} describes the distribution of
	pairwise star-collisions, it remains to understand whether larger
	clusters of vertices with identical stars can occur.
	
	The next result shows that, at the critical scaling,
	such higher-order coincidences are suppressed with high probability.
	Consequently, all non-trivial units are essentially generated by
	pairwise collisions, and the total number of non-trivial units inherits
	the Poisson limit from $X_1$.
	\subsubsection{Proof of Theorem~\ref{thm:units-structure}}
\begin{proof}
	Let
	$
	T_n := \sum_{u<v<w}
	\mathbf{1}\{\Gamma(u)=\Gamma(v)=\Gamma(w)\neq \emptyset\}
	$
	denote the number of triples of vertices having identical non-empty stars.
	If there exists a unit of size at least $3$, then there exist distinct
	vertices $u,v,w$ such that
	$
	\Gamma(u)=\Gamma(v)=\Gamma(w)\neq\emptyset.
	$
	Thus,
	$
	\{U_{\ge 3} > 0\} \subseteq \{T_n > 0\},
	$
	where $U_{\ge3} := \sum_{m\ge3} U_m$. Consequently, by Markov’s inequality,
	\[
	\mathbb{P}(U_{\ge3} > 0)
	\le
	\mathbb{P}(T_n > 0)
	\le
	\mathbb{E}[T_n].
	\]
	
Fix distinct vertices $u,v,w$. If
$
\Gamma(u)=\Gamma(v)=\Gamma(w)\neq\emptyset,
$
then every edge in the common star must contain the triple
$\{u,v,w\}$, while every edge containing exactly one or two of
$u,v,w$ must be absent.

Thus, for a fixed integer $r\ge1$, the probability that the common star has
	size exactly $r$ is
	$
	\binom{\binom{n-3}{k-3}}{r}
	p_n^r
	(1-p_n)^{3\binom{n-1}{k-1}+O(n^{k-2})}.
	$
	Summing over all $r\ge1$, we obtain
	\[
	\mathbb{P}\bigl(
	\Gamma(u)=\Gamma(v)=\Gamma(w)\neq\emptyset
	\bigr)
	=
	(1-p_n)^{3\binom{n-1}{k-1}+O(n^{k-2})}
	\sum_{r\ge1}
	\binom{\binom{n-3}{k-3}}{r}p_n^r.
	\]
	Using the binomial theorem,
	\[
	\sum_{r\ge1}
	\binom{\binom{n-3}{k-3}}{r}p_n^r
	=
	(1+p_n)^{\binom{n-3}{k-3}}-1.
	\]
	Since
	\[
	p_n\binom{n-3}{k-3}
	=
	\lambda_n
	\frac{\binom{n-3}{k-3}}{\binom{n-1}{k-1}}
	=
	O\!\left(\frac{\lambda_n}{n^2}\right)
	=
	o(1),
	\]
	we have
	$
	(1+p_n)^{\binom{n-3}{k-3}}-1
	=
	p_n\binom{n-3}{k-3}(1+o(1)).
	$
	Therefore,
	\[
	\mathbb{P}\bigl(
	\Gamma(u)=\Gamma(v)=\Gamma(w)\neq\emptyset
	\bigr)
	=
	p_n\binom{n-3}{k-3}
	e^{-3\lambda_n+o(1)}.
	\]
	Summing over all triples $\{u,v,w\}$ yields
	\[
	\mathbb{E}[T_n]
	=
	\binom{n}{3}
	\binom{n-3}{k-3}
	p_n
	e^{-3\lambda_n+o(1)}.
	\]
	Using
	\[
	\binom{n}{3}
	\binom{n-3}{k-3}
	p_n
	=
	\frac{n(n-1)(n-2)}{6}
	\cdot
	\frac{(n-3)!}{(k-3)!(n-k)!}
	\cdot p_n
	=
	O(n\lambda_n),
	\]
	we obtain
	$\mathbb{E}[T_n]
	=
	O(n\lambda_n e^{-3\lambda_n})$.
	Under the assumption
	$
	\lambda_n
	=
	\tfrac12(\log n+\log\log n)+w,
	$
	we have
	$
	e^{-3\lambda_n}
	=
	n^{-3/2}
	(\log n)^{-3/2}
	e^{-3w},
	$
	and hence
	$$
	n\lambda_n e^{-3\lambda_n}
	=
	O\!\left(
	n^{-1/2}
	(\log n)^{-1/2}
	\right)
	\to0.
	$$
	Therefore,
	$
	\mathbb{E}[T_n]\to0,
	$
	and consequently
	$
	\mathbb{P}(U_{\ge3}>0)\to0,
	$
	which proves part~(i).
	
	\medskip
	\noindent
	For part~(ii), by part~(i), with high probability all non-trivial units
	are of size $2$. In this case there is a one-to-one correspondence between
	non-trivial units and star-collisions of support size $1$, and hence
	$
	Y = X_1
	$
	with probability tending to $1$.
	From Theorem~\ref{thm:r1-poisson}, we have
	\[
	X_1 \xrightarrow{d}
	\mathrm{Poisson}\!\left(
	\frac{k-1}{4}e^{-2w}
	\right).
	\]
	Therefore,
	\[
	Y \xrightarrow{d}
	\mathrm{Poisson}\!\left(
	\frac{k-1}{4}e^{-2w}
	\right),
	\]
	which completes the proof.
\end{proof}
	\subsection{Proofs for the results on the Linear algebraic and dynamical consequence section}
Here we are going to use the already established results, to prove the result stated in the Section~\ref{sec-appl}.
	\subsubsection{Proof of spectral results}
	\label{pf:spectral-results}
	
	We now prove the spectral consequences stated in Section~\ref{sec:application-spec}. 
	The key input is the structural decomposition of star-dependent matrices recalled in 
	Fact~\ref{thm:finite_decomposition}, together with the probabilistic description of 
	non-trivial units obtained in Theorem~\ref{thm:units-structure}. 
	These results allow us to identify the local spectral component as a random finite-dimensional perturbation, whose size is governed by the unit structure.
	
	\begin{proof}[Proof of Theorem~\ref{thm:spectral-poisson}]
		Each non-trivial unit $W$ contributes a local eigenvalue
		$
		\lambda_W = m_{uu}-m_{uv},
		\quad u\neq v\in W,
		$
		with multiplicity $|W|-1$
		(Fact~\ref{fact-eigen-unit}).
		Thus, the spectrum of $M$ consists of the spectrum of the
		unit contraction $\widehat M$ together with a random finite
		multiset of local eigenvalues whose total multiplicity equals
		$\dim(\mathcal H_{\mathrm{loc}})$.
		Recall that from Fact~\ref{thm:finite_decomposition} that
		$\mathcal H_{\mathrm{loc}} = \bigoplus_{W} S_W$,
		where the sum runs over all non-trivial units $W$, and $\dim(S_W)=|W|-1$.
		Therefore, the proof of part-(i) follows from the Corollary~\ref{cor:units-trivial}.
		For the proof of part-(ii), we use
		$\dim(\mathcal H_{\mathrm{loc}})=\sum_W (|W|-1)$.
		In the regime
		$\lambda_n = \tfrac{1}{2}(\log n + \log\log n) + w$,
		Theorem~\ref{thm:units-structure} implies that the number of non-trivial units converges in distribution to
		$\mathrm{Poisson}\!\left(\frac{k-1}{4}e^{-2w}\right)$,
		and that, with high probability, all such units have size $2$.
		Therefore,
		$\dim(\mathcal H_{\mathrm{loc}})
		=
		|\mathcal W|
		\quad \text{with high probability}$,
		and thus
		\[
		\dim(\mathcal H_{\mathrm{loc}})
		\xrightarrow{d}
		\mathrm{Poisson}\!\left(\frac{k-1}{4}e^{-2w}\right).
		\]
		Each non-trivial unit $W$ contributes an eigenvalue of multiplicity $|W|-1$ 
		(Fact~\ref{fact-eigen-unit}). Thus, the spectrum of $M$ consists of the spectrum of the unit contraction $\widehat M$ together with a random finite multiset of additional eigenvalues whose total multiplicity equals $\dim(\mathcal H_{\mathrm{loc}})$.
	\end{proof}
\begin{proof}[Proof of Corollary~\ref{cor:esd-stability}]
	Let
	$
	R_n := M_n|_{\mathcal H_{\mathrm{glob}}},
	\;
	S_n := M_n|_{\mathcal H_{\mathrm{loc}}},$
	and let
   $
	m_n := \dim(\mathcal H_{\mathrm{loc}}).
	$
	Recall the invariant decomposition
	$
	\mathbb{R}^{V(H_n)}
	=
	\mathcal H_{\mathrm{glob}}
	\oplus
	\mathcal H_{\mathrm{loc}}.
	$
	Since both subspaces are $M_n$–invariant, the spectrum of $M_n$
	is the disjoint union of the spectra of $R_n$ and $S_n$. Hence
	the empirical spectral distribution satisfies
	$$
	\mu^{M_n}
	=
	\frac{n-m_n}{n}\mu^{R_n}
	+
	\frac{m_n}{n}\mu^{S_n}.
	$$
	
	By Theorem~\ref{thm:spectral-poisson},
	$m_n$ converges in distribution to a Poisson random variable.
	In particular,
	$
	\frac{m_n}{n}
	\xrightarrow{\mathbb P}
	0.
	$
	Therefore, for every bounded continuous function
	$f:\mathbb R\to\mathbb R$,
	\[
	\left|
	\int f\,d\mu^{M_n}
	-
	\int f\,d\mu^{R_n}
	\right|
	\le
	2\|f\|_\infty \frac{m_n}{n}
	\xrightarrow{\mathbb P}
	0.
	\]
	Thus $\mu^{M_n}$ and $\mu^{R_n}$ have the same weak limit points.
	
	It remains to verify tightness. Since $M_n$ is self-adjoint,
	\[
	\int x^2\, d\mu^{M_n}(x)
	=
	\frac1n \operatorname{Tr}(M_n^2)
	=
	\frac1n \operatorname{Tr}(M_n M_n^*).
	\]
	By assumption,
    $$
	\sup_n
	\mathbb E\!\left[
	\int x^2\, d\mu^{M_n}(x)
	\right]
	<
	\infty.
	$$
	Hence, for every $R>0$,
	\[
	\mathbb E\!\left[\mu^{M_n}(|x|>R)\right]
	\le
	\frac1{R^2}
	\mathbb E\!\left[
	\int x^2\, d\mu^{M_n}(x)
	\right]
	=
	O(R^{-2}),
	\]
	uniformly in $n$. Therefore the sequence
	$\{\mu^{M_n}\}$ is tight in probability, and consequently
	$\{\mu^{R_n}\}$ is also tight, since the two sequences differ by
	a vanishing mass.
	
	It follows that $\mu^{M_n}$ converges weakly if and only if
	$\mu^{R_n}$ converges weakly, and in that case the limits coincide.
\end{proof}
	\subsubsection{Proof of Theorem~\ref{thm:asymptotic_capture_dynamics}}
	
	Fix $n\ge1$ and condition on a realization of $H_n$.
	
	For any hypergraph $H$ and any star-dependent matrix $M$, 
	Fact~\ref{thm:finite_decomposition} yields an invariant decomposition
	$\mathbb R^{V(H)} = \mathcal H_{\mathrm{glob}} \oplus \mathcal H_{\mathrm{loc}}$,
	where both $\mathcal H_{\mathrm{glob}}$ and $\mathcal H_{\mathrm{loc}}$ are invariant under $M$.
	
	Let $\mathcal V \subseteq \mathbb R^{V(H)}$ be any $M$-invariant subspace.
	Since $\mathcal V$, $\mathcal H_{\mathrm{glob}}$, and $\mathcal H_{\mathrm{loc}}$ 
	are invariant, it follows that
	\[
	\mathcal V_{\mathrm{glob}} := \mathcal V \cap \mathcal H_{\mathrm{glob}},
	\qquad
	\mathcal V_{\mathrm{loc}} := \mathcal V \cap \mathcal H_{\mathrm{loc}}
	\]
	are also invariant subspaces. Moreover, by the direct sum decomposition,
	\[
	\mathcal V = \mathcal V_{\mathrm{glob}} \oplus \mathcal V_{\mathrm{loc}}.
	\]
	By construction, $\mathcal V_{\mathrm{glob}} \subseteq \mathcal H_{\mathrm{glob}}$,
	and the restriction of $M$ to $\mathcal H_{\mathrm{glob}}$ is unitarily equivalent
	to the unit contraction $\widehat M$. Hence $\mathcal V_{\mathrm{glob}}$
	is the lift of a $\widehat M$-invariant subspace.
	
	\medskip
	
	We now incorporate the probabilistic structure of $H_n$.
	
	\smallskip
	
	\noindent
	\emph{(i) Regime $\lambda_n=\log n + c$.}
	By Corollary~\ref{cor:units-trivial}, with high probability all units are trivial,
	and hence $\mathcal H_{\mathrm{loc}}=\{0\}$.
	On this event, every invariant subspace satisfies
	$\mathcal V = \mathcal V_{\mathrm{glob}}$, and therefore
	\[
	\mathcal I(M_n)
	=
	\mathrm{Lift}\bigl(\mathcal I(\widehat M_n)\bigr).
	\]
	This proves the first part.
	
	\smallskip
	
	\noindent
	\emph{(ii) Regime $\lambda_n=\tfrac{1}{2}(\log n+\log\log n)+w$.}
	From Theorem~\ref{thm:units-structure}, the number of non-trivial units
	converges in distribution to a Poisson random variable, and with high probability
	all such units have size $2$.
	Let $N_2$ denote the number of units of size $2$, and define
	\[
	R_n
	:=
	\sum_{|W|\ge 3} (|W|-1).
	\]
	Then
	$\dim(\mathcal H_{\mathrm{loc}})
	=
	N_2+R_n$.
	By Theorem~\ref{thm:units-structure},
	$
	N_2
	\Rightarrow
	\mathrm{Poisson}\!\left(\frac{k-1}{4}e^{-2w}\right).
	$
	Moreover, with high probability there exists no unit of size at least $3$. Hence, with probability tending to $1$, the sum defining $R_n$ is empty, and therefore
$	R_n=0$.
	Consequently,
	$R_n\xrightarrow{\mathbb P}0$.
	Therefore, by Slutsky's theorem \cite{van2000asymptotic},
	\[
	\dim(\mathcal H_{\mathrm{loc}})
	=
	N_2+R_n
	\Rightarrow
	\mathrm{Poisson}\!\left(\frac{k-1}{4}e^{-2w}\right).
	\]

	Since $\dim(\mathcal H_{\mathrm{loc}}) \xrightarrow{d} \text{Poisson}(\mu)$, the sequence $\{\dim(\mathcal H_{\mathrm{loc}})\}$ is tight (or stochastically bounded)\cite[Section-5]{billingsley2013convergence}. Formally, a sequence of probability measures is tight if no probability mass ``escapes to infinity" in the limit. In our case, where the random variables take values in $\mathbb{R}_{\geq 0}$, this implies that for every $\varepsilon > 0$, there exists a compact set $K = [0, M]$ such that the probability of the sequence falling outside $K$ is arbitrarily small.The statement is as follows: For every $\varepsilon > 0$, there exists a constant $M > 0$ such that $\limsup_{n \to \infty} \mathbb{P} \big( \dim(\mathcal H_{\mathrm{loc}}) > M \big) < \varepsilon$.

	Since $\mathcal V_{\mathrm{loc}} \subseteq \mathcal H_{\mathrm{loc}}$, we have
	$\mathbb P\!\big(\dim(\mathcal V_{\mathrm{loc}})>M\big)
	\le
	\mathbb P\!\big(\dim(\mathcal H_{\mathrm{loc}})>M\big)$,
	and therefore the same tightness bound holds for $\dim(\mathcal V_{\mathrm{loc}})$.
	Hence the result.
	\qed
	\begin{proof}[Proof of Corollary~\ref{cor:unit-synchronization}]
		The key observation is that the dynamics generated by $M_n$ preserves unit-synchronization, in the sense that if an initial state is constant on each unit, then this property is preserved under iteration. 
		
		In the regime $\lambda_n=\log n+c$, by Theorem~\ref{thm:asymptotic_capture_dynamics}, with high probability all units are trivial, so no non-trivial synchronization clusters exist and $\mathcal H_{\mathrm{loc}}=\{0\}$.
		
		In the regime $\lambda_n=\tfrac{1}{2}(\log n+\log\log n)+w$, non-trivial units appear only in finite number and are asymptotically governed by a Poisson law. Therefore, the number of non-trivial units capable of supporting
		unit-synchronized states converges in distribution to
		$
		\mathrm{Poisson}\!\left(
		\frac{k-1}{4}e^{-2w}
		\right).
		$
	\end{proof}
	\begin{proof}[Proof of Proposition~\ref{thm:asymptotic_fingerprinting}]
		By Corollary~\ref{cor:units-trivial}, if
		$\lambda_n=\log n+c$,
		then with high probability every unit of $H_n$ is a singleton.
		Equivalently, distinct vertices have distinct stars.
	\end{proof}

\end{document}